\documentclass[11pt]{amsart}
\usepackage{mathrsfs,amsmath,amsthm,amsfonts,amssymb,amscd,amsbsy,latexsym,dsfont,stmaryrd,color}
\usepackage{enumerate}
\usepackage{eucal}
\usepackage[T1]{fontenc}
\usepackage[unicode]{hyperref}

\usepackage[all,2cell]{xy} \UseAllTwocells \SilentMatrices
\usepackage{verbatim}

\usepackage[utf8]{inputenc}
\begin{document}

\newcommand{\INVISIBLE}[1]{}


\newtheorem{thm}{Theorem}[section]

\newtheorem{lem}[thm]{Lemma}
\newtheorem{cor}[thm]{Corollary}
\newtheorem{prp}[thm]{Proposition}

\theoremstyle{remark}
\newtheorem{rmk}[thm]{Remark}
\newtheorem{ex}[thm]{Example}

\theoremstyle{definition}
\newtheorem{dfn}[thm]{Definition}
\newtheorem{notation}[thm]{Notation}
\newtheorem{question}[thm]{Question}
\newtheorem{convention}[thm]{Convention}

\newcommand{\aro}{\longrightarrow}
\newcommand{\arou}[1]{\stackrel{#1}{\longrightarrow}}
\newcommand{\linto}{\lhook\joinrel\longrightarrow}
\newcommand{\LR}{\Leftrightarrow}
\newcommand{\RA}{\Longrightarrow}

\newcommand{\mm}[1]{\mathrm{#1}}
\newcommand{\bm}[1]{\boldsymbol{#1}}
\newcommand{\bb}[1]{\mathbf{#1}}

\newcommand{\bA}{\boldsymbol A}
\newcommand{\bB}{\boldsymbol B}
\newcommand{\bC}{\boldsymbol C}
\newcommand{\bD}{\boldsymbol D}
\newcommand{\bE}{\boldsymbol E}
\newcommand{\bF}{\boldsymbol F}
\newcommand{\bG}{\boldsymbol G}
\newcommand{\bH}{\boldsymbol H}
\newcommand{\bI}{\boldsymbol I}
\newcommand{\bJ}{\boldsymbol J}
\newcommand{\bK}{\boldsymbol K}
\newcommand{\bL}{\boldsymbol L}
\newcommand{\bM}{\boldsymbol M}
\newcommand{\bN}{\boldsymbol N}
\newcommand{\bO}{\boldsymbol O}
\newcommand{\bP}{\boldsymbol P}
\newcommand{\bY}{\boldsymbol Y}
\newcommand{\bS}{\boldsymbol S}
\newcommand{\bX}{\boldsymbol X}
\newcommand{\bZ}{\boldsymbol Z}

\newcommand{\ccc}[1]{\mathcal{#1}}
\newcommand{\cc}[1]{\mathscr{#1}}

\newcommand{\ca}{\cc{A}}

\newcommand{\cb}{\cc{B}}

\newcommand{\cC}{\cc{C}}

\newcommand{\cd}{\cc{D}}

\newcommand{\ce}{\cc{E}}

\newcommand{\cf}{\cc{F}}

\newcommand{\cg}{\cc{G}}

\newcommand{\ch}{\cc{H}}

\newcommand{\ci}{\cc{I}}

\newcommand{\cj}{\cc{J}}

\newcommand{\ck}{\cc{K}}

\newcommand{\cl}{\cc{L}}

\newcommand{\cm}{\cc{M}}

\newcommand{\cn}{\cc{N}}

\newcommand{\co}{\cc{O}}

\newcommand{\cp}{\cc{P}}

\newcommand{\cq}{\cc{Q}}

\newcommand{\cR}{\cc{R}}

\newcommand{\cs}{\cc{S}}

\newcommand{\ct}{\cc{T}}

\newcommand{\cu}{\cc{U}}

\newcommand{\cv}{\cc{V}}

\newcommand{\cy}{\cc{Y}}

\newcommand{\cw}{\cc{W}}

\newcommand{\cz}{\cc{Z}}

\newcommand{\cx}{\cc{X}}

\newcommand{\g}[1]{\mathfrak{#1}}

\newcommand{\af}{\mathds{A}}
\newcommand{\PP}{\mathds{P}}

\newcommand{\GL}{\mathrm{GL}}
\newcommand{\PGL}{\mathrm{PGL}}
\newcommand{\SL}{\mathrm{SL}}
\newcommand{\NN}{\mathds{N}}
\newcommand{\ZZ}{\mathds{Z}}
\newcommand{\CC}{\mathds{C}}
\newcommand{\QQ}{\mathds{Q}}
\newcommand{\RR}{\mathds{R}}
\newcommand{\FF}{\mathds{F}}
\newcommand{\DD}{\mathbf{D}}
\newcommand{\VV}{\mathds{V}}
\newcommand{\HH}{\mathds{H}}
\newcommand{\MM}{\mathds{M}}
\newcommand{\OO}{\mathds{O}}
\newcommand{\LL}{\mathds L}
\newcommand{\BB}{\mathds B}
\newcommand{\kk}{\mathds k}
\newcommand{\bs}{\mathbf S}
\newcommand{\GG}{\mathds G}

\newcommand{\al}{\alpha}

\newcommand{\be}{\beta}

\newcommand{\ga}{\gamma}
\newcommand{\Ga}{\Gamma}

\newcommand{\om}{\omega}
\newcommand{\Om}{\Omega}

\newcommand{\vte}{\vartheta}
\newcommand{\te}{\theta}
\newcommand{\Te}{\Theta}

\newcommand{\ph}{\varphi}
\newcommand{\Ph}{\Phi}

\newcommand{\ps}{\psi}
\newcommand{\Ps}{\Psi}

\newcommand{\ep}{\varepsilon}

\newcommand{\vr}{\varrho}

\newcommand{\de}{\delta}
\newcommand{\De}{\Delta}

\newcommand{\la}{\lambda}
\newcommand{\La}{\Lambda}

\newcommand{\ka}{\kappa}

\newcommand{\si}{\sigma}
\newcommand{\Si}{\Sigma}

\newcommand{\ze}{\zeta}

\newcommand{\red}{{\rm red}}
\newcommand{\fr}[2]{\frac{#1}{#2}}
\newcommand{\vs}{\vspace{0.3cm}}
\newcommand{\na}{\nabla}
\newcommand{\pd}{\partial}
\newcommand{\po}{\cdot}
\newcommand{\met}[2]{\left\langle #1, #2 \right\rangle}
\newcommand{\ev}{{\rm ev}}

\newcommand{\pos}[2]{#1\llbracket#2\rrbracket}

\newcommand{\cpos}[2]{#1\langle#2\rangle}

\newcommand{\id}{\mathrm{id}}

\newcommand{\unity}{\mathds 1}

\newcommand{\ti}{\times}
\newcommand{\tiu}[1]{\underset{#1}{\times}}

\newcommand{\ot}{\otimes}
\newcommand{\otu}[1]{\underset{#1}{\otimes}}

\newcommand{\wh}{\widehat}
\newcommand{\wt}{\widetilde}
\newcommand{\ov}[1]{\overline{#1}}
\newcommand{\un}[1]{\underline{#1}}

\newcommand{\op}{\oplus}

\newcommand{\lid}{\varinjlim}
\newcommand{\lip}{\varprojlim}

\newcommand{\rep}[2]{\mm{Rep}_{#1}(#2)}
\newcommand{\repo}[2]{\mm{Rep}^\circ_{#1}(#2)}
\newcommand{\modules}[1]{#1\text{-}\mathbf{mod}}
\newcommand{\Modules}[1]{#1\text{-}\mathbf{Mod}}
\newcommand{\dmod}[1]{\mathcal{D}(#1)\text{-}{\bf mod}}

\newcommand{\hh}[3]{\mm{Hom}_{#1}\left(#2,#3\right)}
\newcommand{\ee}[2]{\mm{End}_{#1}(#2)}

\newcommand{\spc}{\mathrm{Spec}\,}
\newcommand{\spf}{\mathrm{Spf}\,}

\newcommand{\ega}[3]{[EGA $\textsc{#1}_{#2}$, #3]}

\newcommand{\asts}{\vspace{.2cm}\begin{center}***\end{center}\vspace{.2cm}}

\title[Finite torsors]{Finite torsors on projective schemes defined over a discrete valuation ring
}
\date{\today}

\author[P. H. Hai]{Ph\`ung H\^o Hai}

\address{Institute of Mathematics, Vietnam Academy of Science and Technology, Hanoi, 
Vietnam}

\email{phung@math.ac.vn}

\author[J. P. dos Santos]{Jo\~ao Pedro  dos Santos}

\address{Institut de Math\'ematiques de Jussieu -- Paris Rive Gauche, 4 place Jussieu, 
Case 247, 75252 Paris Cedex 5, France}

\email{joao\_pedro.dos\_santos@yahoo.com}

\subjclass[2010]{14F10,	14L15 }

\begin{abstract}
Given a Henselian and Japanese discrete valuation ring $A$ and a flat and projective $A$-scheme $X$, we follow the approach of \cite{biswas-dos_santos11} to introduce  a {\it full subcategory} of coherent modules on $X$ which is then shown to be Tannakian. We then prove that, under normality of the generic fibre, the associated affine and flat group is pro-finite in a strong sense (so that its ring of functions is a Mittag-Leffler $A$-module) and that it classifies finite torsors $Q\to X$. This establishes an analogy to Nori's theory of the essentially finite fundamental group.  In addition, we compare our theory with the ones recently developed by Mehta-Subramanian and Antei-Emsalem-Gasbarri. Using the comparison with the former, we show that any quasi-finite torsor $Q\to X$ has a reduction of structure group to a finite one.  
\end{abstract}

\maketitle





\section{Introduction}

Let $A$ be a discrete valuation ring and $X$ a projective flat $A$-scheme carrying an $A$-point $x_0$. In recent times, a certain number of mathematicians has proposed    constructions of an affine and flat group scheme $\Pi(X,x_0)$ over $A$ with the distinctive property that morphisms to finite and flat group schemes $\Pi(X,x_0)\to G$ should canonically correspond to pointed $G$-torsors over $X$. See \cite{gasbarri03}, \cite{antei-emsalem-gasbarri18} and \cite{mehta-subramanian13}. 
These theories are of course the analogues of  Nori's   \cite{nori}---developed in the case of a base {\it field}---and as such, might be treated by similar mechanisms. Let us recall that Nori's theory can be developed through
\begin{enumerate}\item[({\bf SS})] a Tannakian category of semi-stable vector bundles on $X$ \cite{nori}, through
\item[({\bf F})] the construction of fibre products of torsors   \cite[Chapter 2]{nori82}, 
or through 
\item[({\bf T})] a Tannakian category of vector bundles which are trivialized by proper and surjective morphisms   \cite{biswas-dos_santos11}.
\end{enumerate} The works \cite{gasbarri03} and \cite{antei-emsalem-gasbarri18} adopt point of view $({\bf F})$, whereas \cite{mehta-subramanian13} opts for a variant of ({\bf SS}). In the present paper we focus   on  ({\bf T}) to construct the affine and flat group scheme $\Pi(X,x_0)$.
The advantage  of this approach is that we can realise $\rep {A}{\Pi(X,x_0)}$ as a {\it full} abelian subcategory of $\bb{coh}(X)$ so that, not only we can construct the affine group scheme   parametrising torsors, but we can naturally regard its category of representations in geometric terms. This facet is missing in \cite{antei-emsalem-gasbarri18} (since it is not understood what coherent modules are obtained by twisting representations via the fundamental torsor) as well as in \cite{mehta-subramanian13} (since the authors there focus solely on representations on free $A$-modules).  
On the downside, it is true that if no condition on the singularities of $X$  is imposed,  
then our approach ends up containing {\it too much} geometric information in the sense that it may account for torsors which are not necessarily finite over $X$; this is well understood in the case of non-normal schemes  \cite[Example 6]{biswas-dos_santos12}.  

Let us now review the remaining sections of the paper. Section \ref{subsidiary_material} serves to gather some simple facts from   algebraic  geometry (Lemmas \ref{03.07.2018--1}, \ref{03.07.2018--2} and \ref{03.05.2018--1}), to fix   notation for a class of morphisms which is used allover in the paper (see Definition \ref{26.06.2018--1}) and to   put forward  criteria allowing to decide when a morphism of group schemes is faithfully flat (Lemma \ref{29.11.2018--1} and Lemma \ref{21.03.2019--1}). 

It is in Section \ref{14.04.2018--3} where our theory starts to take form. In it, we introduce the central notion of modules which are {\it relatively trivial} (Definition \ref{20.02.2019--3}) and the category of modules which are relatively trivialized by a proper and surjective morphism (Definition \ref{05.09.2018--3}). The first useful properties of the category of such modules are also developed in Section \ref{14.04.2018--3}:  
we show how to find more convenient trivializations (see Lemma \ref{20.02.2019--4}) and how to make the first steps towards controlling kernels and cokernels of morphisms between relatively trivial modules (Proposition \ref{05.04.2018--3}). 

In Section \ref{04.07.2018--1}, the work initiated in Section \ref{14.04.2018--3} is taken further ahead and we show that---under mild assumptions on the base scheme $X$---the category of modules which become relatively trivial after being pulled back along a suitable  morphism $\ph:Y\to X$, call it $\g T_\ph$,  is in fact abelian, see Theorem \ref{06.04.2018--1}. In addition, concentrating on the objects in $\g T_\ph$ which can be ``dominated'' by locally free ones and using the fibre at an $A$-point $x_0$ of $X$, we   construct a neutral Tannakian category (terminology is that of \cite[Definition 1.2.5]{duong-hai18}) over $A$. We consequently obtain an affine and flat group scheme $\Pi(X,\ph,x_0)$ whose category of representations is naturally a full subcategory of $\bb{coh}(X)$, see Definition \ref{14.01.2019--2}. 

If $K$ stands for the field of fractions of $A$ and $X\ot_AK$ is normal, Section \ref{11.04.2019--4} explains how to obtain information about $\Pi(X,\ph,x_0)\ot_AK$ from the category of essentially finite vector bundles (as defined \cite{nori}) on $X\ot_AK$. See Theorem \ref{11.04.2019--2} and Corollary \ref{16.04.2019--1}. 

Section \ref{14.03.2019--1} studies   $\g T_\ph$   by means of an idea introduced in \cite{hai-dos_santos18}. In  this work, we found a property of representations of a flat and affine group scheme $G$ over a complete $A$, called prudence, which allows us to verify when the $A$-module of functions  $A[G]$ is   free. This is particularly pleasing for the present theory since, on the geometric side, prudence is roughly Grothendieck's algebraization. This  fact   allows us to show that the full Galois groups constructed from $\g T_\ph$  are in many cases {\it finite} group schemes, see Corollary \ref{25.06.2018--2}. 
In doing so, we provide an argument to substantiate a claim made in Lemma 3.1 and Theorem 4.1 of \cite{mehta-subramanian13}. (See also our summary of Section \ref{06.03.2019--1} below.)

Section \ref{RFT}  recalls the reduced fibre theorem and puts it   in a way which can  conveniently be employed in the rest of the work, see Theorem \ref{reduced_fibre_theorem}.   
Note that, this powerful theorem comes with the (mild) hypothesis that a certain generic fibre should be geometrically reduced; since we wish to profit  from this result even in the absence of such an assumption, we benefit from this accessory section to explain how to get rid of schemes which fail to be geometrically reduced, see Lemma \ref{25.10.2018--1}. It is perhaps worth pointing out that it is at this point that  additional  properties of   $A$---that it be Henselian and Japanese---start being required.

Recall that in Section \ref{04.07.2018--1} we associated to a   morphism $\ph:Y\to X$ (satisfying certain assumptions) tensor categories $\g T_\ph$ which, once polished, become neutral Tannakian ones. 
In Section \ref{FGS} we put these categories together by employing the results of Section \ref{RFT}, this is made precise by   Theorem \ref{05.11.2018--1}. We are then able to construct  a   neutral Tannakian category $\g T_X$  which is a full subcategory of $\bb{coh}(X)$, see Definition \ref{30.10.2018--1}. The group scheme associated to $\g T_X$ via the fibre functor $\bullet|_{x_0}:\g T_X\to\modules A$, call it $\Pi(X,x_0)$,   has the distinctive property that its category of representations $\rep A{\Pi(X,x_0)}$ is equivalent to the full subcategory $\g T_X$  of $\bb{coh}(X)$. Under the assumption of normality, we are then able to show that $\Pi(X,x_0)$ is pro-quasi-finite (in the sense of Definition \ref{25.02.2019--1}), see Theorem \ref{15.03.2019--1}. If $A$ is in addition complete,  the results of Section \ref{14.03.2019--1} apply directly and we obtain that  $\Pi(X,x_0)$ is {\it strictly} pro-finite (Theorem \ref{15.03.2019--1}). Finally, an indirect argument  allows us to deduce from this that $\Pi(X,x_0)$ is {\it strictly} pro-finite even if $A$ fails to be complete: see Theorem \ref{25.03.2019--1}. As a consequence, the ring of functions of $\Pi(X,x_0)$ is a Mittag-Leffler $A$-module (Corollary \ref{16.04.2019--2}). 

In Section \ref{20.03.2019--1} we make a brief digression to exhibit   examples and applications of the preceding theory to the case of invertible sheaves, see Proposition \ref{11.01.2019--1}. These examples were a constant source of guidance while elaborating our rhetoric and we believe they shall be useful to the reader.

Section \ref{15.03.2019--2} sets out to give conditions for a finite group scheme---appearing as the structure group of a torsor over $X$---to be a quotient of $\Pi(X,x_0)$, see Proposition \ref{29.11.2018--2} and Corollary \ref{27.03.2019--1}.

In Section \ref{06.03.2019--1} we review parts of  \cite{mehta-subramanian13}. Let us explain the reason for offering such a revision before summarising the contents of Section \ref{06.03.2019--1}. 
If $X_0$ stands for the special fibre of $X$, one of the main points of \cite{mehta-subramanian13} is  to use a technique in \cite{deninger-werner05}, see pages 574 to 576 there, to show, speaking colloquially, how to trivialise a vector bundle  $E\in\bb{VB}(X_0)$ on the special  fibre of a {\it flat} $A$-scheme $Y$ endowed with a finite morphism to $X$, see the {\it proof} of  Lemma 3.1       in \cite{mehta-subramanian13}. In our opinion, \cite{mehta-subramanian13} only offers a sketch of how to implement this  beautiful idea: in the proof of \cite[Lemma 3.1]{mehta-subramanian13} the necessary ``base-change'' to an extension of $A$ is not mentioned, neither is  \cite[Theorem 17]{deninger-werner05}. In addition, as far as we can see, the proof of Theorem 4.1 in \cite{mehta-subramanian13}---which should extend Lemma 3.1 of op. cit.---offers solely a construction of a smooth curve inside  a projective and smooth $A$-scheme. Given all these considerations, we set out to put the method of Deninger and Werner in a more robust form. In doing so we are able to 
circumvent Liu's theorem on the existence of semi-stable models of relative curves (a crucial point in \cite{deninger-werner05} and hence in \cite[Lemma 3.1]{mehta-subramanian13})
by means of the  reduced fibre theorem (Theorem \ref{reduced_fibre_theorem}). The advantage is that the latter result holds for more general schemes than just curves. This allows us to show 
Theorem \ref{14.05.2018--1} saying that an $F$-trivial  vector bundle $E$ (see the introduction of Section \ref{06.03.2019--1} for the definition) on $X_0$ becomes relatively trivial (Definition \ref{20.02.2019--3}) after   pull back by a finite morphism $Y\to X$ from a flat $A$-scheme. In the same vein, we offer Theorem \ref{MS_theorem} explaining how to trivialize any  vector bundle  $\ce\in\bb{VB}(X)$ with essentially finite fibres by  a finite morphism $Y\to X$ from a flat $A$-scheme.  (It should be noted that the proof of this result is technically intricate, but the     main idea comes straight from Theorem \ref{14.05.2018--1}.) The section then ends with a direct comparison between the category $\g T_X$ and the one studied by \cite{mehta-subramanian13}, see Corollary \ref{23.11.2018--2}, where it is shown that vector bundles on $\g T_X$ agree with those introduced in op. cit. 

Section \ref{14.01.2019--3} gives further applications to the theory of torsors and shows that torsors under quasi-finite group schemes actually come from finite ones, see Theorem \ref{25.06.2018--1}. Another salient point developed in this section is the comparison to the theory of \cite{antei-emsalem-gasbarri18}: 
By demonstrating that $\Pi(X,x_0)$ ``classifies'' pointed quasi-finite torsors over $X$ (Theorem \ref{13.03.2019--2}) we prove that $\Pi(X,x_0)$ agrees with the group denoted $\pi_1(X,x_0)^{\rm qf}$ in op.cit., see Theorem \ref{13.03.2019--1}.

\section*{Notations and conventions}\label{notations_conventions}

\subsection*{{On the base ring.}}

\begin{enumerate}[(1)]
\item We let $A$ be a discrete valuation ring with uniformizer $\pi$, field of fractions $K$ and residual field $k$. 
The quotient ring $A/(\pi^{n+1})$ is denoted by $A_n$. 
\item Given an object $W$ over $A$ (a scheme, a module, etc) and an $A$-algebra $B$,  we find useful to write $W_B$ instead of $W\ot_AB$. If context prevents any misunderstanding, we also employ $W_n$ instead of $W_{A_n}$.

\item The characteristic of a discrete valuation ring is the couple $(r,p)$, where $r$ is the characteristic of the field of fractions and $p$ that of the residue field. 
\end{enumerate}

\subsection*{ On general algebraic geometry}
\begin{enumerate}[(1)]
\item A {\it vector bundle} over a scheme is a locally free sheaf of finite rank. A vector bundle is said to be trivial if it is isomorphic to a direct sum of a copies of the structure sheaf (in particular the rank is constant).

\item Let $R$ be a noetherian ring and $Y$ a proper $R$-scheme. We shall say that $Y$ is {\it $H^0$-flat over $R$} if it is flat and cohomologically flat \cite[p.206]{bosch-lutkebohmert-raynaud90} in degree zero over $\spc R$. In this work we shall employ constantly that if $Y$ is $R$-flat, then (a)  $H^0$-flatness amounts to exactness of $M\mapsto H^0(\co_Y\ot_AM)$ \cite[III.12.5 and III.12.6]{hartshorne77} and 
(b) if $Y$ the fibres of $Y$      are geometrically reduced \ega{IV}{2}{4.6.2, p.68}, then $Y$ is $H^0$-flat over $R$ \ega{III}{2}{7.8.6, p.206}.

\item If $R$, respectively $T$, is an $\FF_p$-algebra, respectively $\FF_p$-scheme, we write $F_R:R\to R$, respectively $F_T:T\to T$, to denote the {\it Frobenius morphism}. If $T$ is in addition a scheme over a perfect field, we adopt the notations of \cite[Part I, 9.1]{jantzen87} with the exception that we write ${\rm Fr}^s:M^{(t)}\to M^{(s+t)}$ while Jantzen uses $F_{M^{(t)}}^s:M^{(t)}\to M^{(t+s)}$.

\item If $R$ is a discrete valuation ring and $y:\spc R\to Y$ is an $R$-point of a scheme $Y$, we shall write $y_{\rm gen}$ for the image of the generic point of $\spc R$ in $Y$.

\item Let $X\to S$ be a morphisms of schemes possessing in addition a section $x:S\to X$. Given morphisms  $\ph:X'\to X$ and $f:S'\to S$,  any morphism of $S$-schemes $x':S'\to X'$ such that $\ph x'=x  f$ shall be called an $S'$-point of $X'$ above $x$. (That is, for the sake of economy we choose not to make base-changes to $S'$.)

\end{enumerate}

\subsection*{ On group schemes}
\begin{enumerate}[(1)]
\item
To avoid repetitions,  ``\emph{group scheme}'' is a synonym for  ``\emph{affine group scheme}.'' If $G$ is a group scheme over a ring $R$,  and $R'$ is an $R$-algebra, we write $R'[G]$ instead of $\co(R'\ot_RG)$.  
\item The category of flat group schemes over a ring $R$ is denoted by $(\mathbf{FGSch}/R)$. 

\item Let $G$ be a flat group scheme over the noetherian ring $R$. When dealing with representations of  $G$ we follow the conventions of \cite[Part I, Ch. 2]{jantzen87} with the exception that the word ``representation'' is reserved for $G$-modules which are of finite type over $R$.  The category of representations is denoted by $\rep RG$. The full subcategory of  representations which underlie  locally free $R$-modules is denoted by ${\rm Rep}_R^\circ(G)$. 
\item The {\it right-regular}, respectively {\it left-regular}, $G$-module \cite[2.7]{jantzen87} shall be denoted by $R[G]_{\rm right}$, respectively $R[G]_{\rm left}$. 
\item If $f:G\to H$ is an arrow of $(\bb{FGSch}/R)$, we let $f^\#:\rep RH\to\rep RG$ be the restriction functor. 
\end{enumerate}

\subsection*{  On torsors }
\begin{enumerate}[(1)]
 
\item Let $R$ be a ring, $X$ an $R$-scheme, 
$G$ and $H$  group schemes over $R$,   $P
\to X$ a $G$-torsor and $Q\to X$ an $H$-torsor. A {\it generalized morphism} from $P$ to $Q$ is a couple $(f,\rho)$ consisting of an arrow of $X$-schemes $f:P\to Q$ and a morphism of group schemes $\rho:G\to H$ such that, for points with value on arbitrary $R$-algebras, we have  $f(yg)=f(y)\rho(g)$. In this case, we say that the generalized morphism {\it covers} the morphism $\rho$. If $G=H$, then a morphism of torsors from $P$ to $Q$ is simply a $G$-equivariant morphism of $X$-schemes or a generalized morphism covering the identity. 

\item If $R$ is a ring,  $\rho:G\to H$ is an arrow of $(\bb{FGSch}/R)$ and  $P\to X$ is a $G$-torsor, we let $P\ti^GH$ or $P\ti^\rho H$ be the associated $H$-torsor (see \cite[III.4.3.2, p.368]{demazure-gabriel70} or \cite[Part I, 5.14]{jantzen87}). 

\item If $G$ is a flat group scheme over a ring $R$, $P\to X$ is a $G$-torsor and $M$ is a representation of $G$, then we let $\te_P(M)$ stand for the coherent sheaf constructed by twisting $P$ by $M$, see \cite[Part 1, 5.8--9]{jantzen87}, where it is denoted by $\cl_{P/G}(M)$.
\end{enumerate}

\subsection*{Miscellaneous}
\begin{enumerate}[(1)]
\item All tensor categories and functors are to be taken in the sense of \cite[\S1]{deligne-milne82}. Let $(\g C,\ot)$ be a rigid tensor category \cite[Definition 1.7]{deligne-milne82}. If $\g C$ is additive, then, for  $a=(a_1,\ldots,a_m)$ and $b=(b_1,\ldots,b_m)$ in $\NN^m$, and $E\in\g C$, we write $\bb T^{a,b}E$ to denote the object $\bigoplus_iE^{\ot a_i}\ot \check{E}^{\ot b_i}$. If $\g C$ is in addition abelian, we let $\langle E\,;\,\g C\rangle_\ot$ stand for the full subcategory of $\g C$ whose objects are subquotients of some  $\bb T^{a,b}E$.  

\item If $X$ is a proper and reduced scheme over a field, we let $\bb{EF}(X)$
 denote the category of essentially finite vector bundles on $X$. See \cite[Definition, p.37-8]{nori}.  
\end{enumerate}

\subsection*{Acknowledgements}The research of PHH is supported by the Vietnam National Foundation for Science and Technology Development  (Nafosted) grant number  101.04-2016.19. Parts of his research on this work were carried out during several visits to France supported by the project “International Associated Laboratory FORMATH VIETNAM” of VAST, grant number QTFR01.04/18-19.,  the Institute Henri Poincar\'e, Paris, and  a ``Poste Rouge'' grant from the CNRS, France.
JPdS thanks A. Ducros for pointing out \cite{temkin10} and for useful discussions, Q. Liu for replying conscientiously to his e-mails concerning cohomological flatness, and I. Biswas as well as M. Romagny for answering some of his questions. 


\section{Subsidiary material}\label{subsidiary_material}

We collect here some simple facts which are useful in developing our arguments.  

\begin{lem}\label{03.07.2018--1}Let $Y$ be a proper, reduced and connected scheme over the field $K$. If $Y$ has a $K$-point then $K=H^0(\co_Y)$.
\end{lem}
\begin{proof}We know that $H^0(\co_Y)=R_1\ti\cdots\ti R_m$, where each $R_i$ is a local Artin algebra. Since $Y$ is connected, we must have $m=1$. So   $H^0(\co_Y)$ is a local Artin algebra. Since $Y$ is reduced, $H^0(\co_Y)$ is reduced, so $H^0(\co_Y)$ is a finite field extension of $K$. The existence of a $K$-rational point determines a morphism of $K$-algebras $H^0(\co_Y)\to K$, and this forces $H^0(\co_Y)$ to be $K$.
\end{proof}

\begin{lem}\label{03.07.2018--2}Let $Y$ be a flat and  proper $A$-scheme. If $H^0(\co_{Y\ot_AK})=K$, then $H^0(\co_Y)=A$. In particular, if 
$Y\ot_AK$ is reduced and connected, and has a $K$-point,  then $H^0(\co_Y)=A$. 
\end{lem}

\begin{proof}We know that $H^0(\co_Y)$ is a finite $A$-algebra and that $H^0(\co_Y)\ot K=K$ by flat base-change \cite[III.9.3, p. 255]{hartshorne77} together with the assumption. Then, $A\subset H^0(\co_Y)\subset K$. Since $A$  is normal, we must have $A=H^0(\co_Y)$. The verification of the last claim follows immediately from Lemma \ref{03.07.2018--1}.
\end{proof}

\begin{lem}\label{03.05.2018--1}Let  $Y$ be a proper and $H^0$-flat $A$-scheme. Let $\cf$ be a vector bundle on $Y$ whose base-changes $\cf_K$ and $\cf_k$ are trivial. Then $\cf$ is trivial. 
\end{lem}

\begin{proof}It is not difficult to see that the function  $y\mapsto\mm{rank}\,\cf_y$ has to be constant and equal to $n$, say.  
From the equalities 
\[\begin{split}
\dim_kH^0(Y_k,\cf_k)&=\dim_kH^0(Y_k,\co_{Y_k})^{\op n}\\
&=\dim_KH^0(Y_K,\co_{Y_K})^{\op n}\\
&=\dim_KH^0(Y_K,\cf_K),
\end{split} 
\]
and Corollary 2 of \cite[\S5]{mumford70}, we conclude that  the canonical morphism 
\[H^0(Y,\cf)\aro H^0(Y_k,\cf_k)
\]
is surjective. Pick sections $s_1,\ldots,s_n\in H^0(Y,\cf)$ such that $\{s_{i}|_{Y_k}\}_{i=1}^n$ is a basis of $\cf_k$. Then, for each point $y$ of the closed fibre, Nakayama's Lemma shows that $\{s_{i,y}\}_{i=1}^n$ is a basis of the free $\co_{y}$-module $\cf_y$ so that   $\mm{Supp}(\cf/\sum\co_Ys_i)$ is a closed subset of $Y$ disjoint from the special fibre. Such a property is only possible if $\cf=\sum\co_Ys_i$. The proof is then complete once we note that, for a $y\in  Y$, the $n$ generators $\{s_{i,y}\}_{i=1}^n$ of $\cf_y\simeq\co_y^{\op n}$ do not admit any non-trivial relation.  
\end{proof}

In organising our findings, we shall make repeated use of a certain class of morphisms of proper schemes. In order to avoid repetitions and to serve as a reference for the reader, we put forward:

\begin{dfn}\label{26.06.2018--1}Let   $X$ be a connected, proper and flat $A$-scheme carrying an $A$-point $x_0$. 
Let $\g S(X,x_0)$ (or simply $\g S$ if context prevents any misunderstanding) be the set of all  $X$-schemes $\ph:Y\to X$ such that:

\begin{enumerate}
\item[$\g S1$.] $\ph$ is proper and surjective. 
\item[$\g S2$.] $Y$ is $H^0$-flat over $A$. 
\item[$\g S3$.] The canonical arrow $Y\to\spc {H^0(\co_Y)}$ admits a section $y_0$ such that $\ph y_0$ extends $x_0$. (There exists an $H^0(\co_Y)$-point in $Y$ above $x_0$.) 
\end{enumerate}
We shall denote by $\g S^+(X,x_0)$ (or simply $\g S^+$) the subset of morphisms $\ph:Y\to X$ in $\g S$ which, in addition, satisfy 
\begin{enumerate}\item[$\g S4$.] The canonical morphism $Y\to\spc H^0(\co_Y)$ is   flat. 
\end{enumerate}
\end{dfn}


\begin{rmk}\label{05.02.2019--1}Let $\ph:Y\to X$         belong  to $\g S(X,x_0)$ as in Definition \ref{26.06.2018--1}; it turns out that in this case, connectedness of $X$ is automatic because of  $\g S3$. Indeed, let $y_0$ be the alluded $H^0(\co_Y)$-point of $Y$ above $x_0$. 
Suppose that $e\in H^0(\co_X)$ is   idempotent and  that $x_0^\#(e)=0$. Then, $\ph^\#(e)\in H^0(\co_Y)$ is such that $y_0^\#(\ph^\#(e))=0$ and this implies that $\ph^\#(e)=0$. Since $\ph$ is surjective, we conclude that $e=0$ (since $e$ has to vanish on each local ring of $X$). Hence, the only connected component of $X$ must be the one containing the image of $x_0$. 
\end{rmk}

We should now gather material on affine group schemes. 

\begin{dfn}\label{22.03.2019--1}Let $G\in(\bb {FGSch}/A)$. We say that $G$ is pseudo-finite if both its fibres are finite group schemes over the respective residue fields. 
\end{dfn}

\begin{rmk}We have no examples of pseudo-finite group schemes which are not quasi-finite to offer. \end{rmk}

As Lemma \ref{29.11.2018--1} below recalls, one advantage of finite group schemes over a field is that the standard criterion for verifying when a morphism is faithfully flat in terms of representation categories \cite[Proposition 2.21, p.139]{deligne-milne82} admits a considerable simplification. This is then transmitted to pseudo-finite group schemes   as argued by Lemma \ref{21.03.2019--1}. 
\begin{lem}[{\cite[pp. 87-8]{nori82}}]\label{29.11.2018--1} Let $u:H\to G$ be a morphism of  group schemes over a field and let \[H\stackrel q\aro  I\stackrel i\aro G\] be its factorisation into a  faithfully flat morphism $q$ and a closed immersion $i$. Assuming that $I$ is finite (which is the case if either $G$ or $H$ is finite) then a necessary and sufficient condition for $u$ to be faithfully flat is that $\dim\co(G)^H=1$.
\end{lem} 
\begin{proof} We leave the proof of necessity  in the statement to the reader and from now on assume that $\dim\co(G)^H=1$.  Clearly,  the equality $I=G$ is equivalent to faithful flatness of $u$. 
Since $\co(G)^H=\co(G)^I$, the assumptions translate into  $\dim\co(G)^I=1$.  Now, we know that $\co(G)$ is a projective $\co(G)^I$-module whose rank equals $\dim\co(I)$ (see \cite[\S11, Theorem 1(B), p.111]{mumford70} or III.2.4 of \cite{demazure-gabriel70}). Then, since $\dim\co(G)^I=1$, we conclude 
that   $\dim\co(I)=\dim\co(G)$, which means that $I= G$. 
\end{proof}

\begin{lem}\label{21.03.2019--1}
Let $u:H\to G$ be an arrow in $(\bb{FGSch}/A)$ and assume that 
$G$ is pseudo-finite. 
Then the following conditions are equivalent. 
\begin{enumerate}[(1)]\item Both equalities  
\[
A\po1= \left(A[G]_{\rm right}\right)^H \quad \text{and}\quad k\po1= \left(k[G]_{\rm right}\right)^H, 
\] are true. 
\item 
The morphism $u$ is faithfully flat. 

\item Let 
\[
\g s(G)=
\left\{ \begin{array}{c} \text{$V\subset A[G]_{\rm right}$ is $G$-invariant, is finitely generated }
\\
\text{ as an $A$-module and contains the constants}
\end{array}
\right\}
\]
and 
\[
\g s_0(G)=
\left\{ \begin{array}{c} \text{$M\subset k[G]_{\rm right}$ is $G$-invariant, is finitely generated }
\\
\text{ as a $k$-space and contains the constants}
\end{array}
\right\}. 
\]
Then, for each $V\in\g s(G)$ and $M\in\g s_0(G)$ we have 
\[A\po1= V^H \quad \text{and}\quad k\po1= M^H. 
\] 
\end{enumerate}
\end{lem}

\begin{proof}

(1) $\Rightarrow$ (2).  The assumptions show that 
\[
\left(K[G]_{\rm right} \right)^{H_K} = K\quad\text{and}\quad \left(k[G]_{\rm right} \right)^{H_k}=k.
\]
By Lemma \ref{29.11.2018--1} we conclude   that $u_K:H_K\to G_K$ and  $u_k:H_k\to G_k$ are faithfully flat. Because of \cite[Theorem 4.1.1]{duong-hai18}, $u$ is faithfully flat.

(2)$\Rightarrow$(3).  This is trivial. 

(3)$\Rightarrow$(1). Any $a\in  A[G]_{\rm right}$, respectively $b\in k[G]_{\rm right}$, belongs to a certain $V\in\g s(G)$, respectively $M\in\g s_0(G)$ because of ``local finiteness'' \cite[1.5, Corollary, p.40]{serre68}. The conclusion then follows.  
\end{proof}

\section{Modules trivialized by a proper and $H^0$-flat scheme}\label{14.04.2018--3}

In this section we introduce the category of coherent sheaves on which all further developments hinge: the category of sheaves which became ``trivial'' after a pull-back by a proper morphisms. Since we wish to work with schemes over a d.v.r., the notion of triviality of a coherent module becomes more subtle than the one over a field, and we need to account for modules coming from the base-ring. This is a source of difficulty specially because the ``base'' of the scheme effectuating the trivialisation might grow.


\begin{dfn}\label{20.02.2019--3}Let $Y$ be any $A$-scheme. We say that $\cf\in\bb{coh}(Y)$ is trivial relatively to $A$ if there exists a coherent sheaf (a finite $A$-module) $F$ such that $\cf=\co_Y\ot_AF$.  
\end{dfn}

\begin{rmk}Note that, if $A$ is a field, then a coherent sheaf on $Y$ is trivial relatively to $A$ if and only if it is trivial in according to our conventions in  Section \ref{notations_conventions}. 
\end{rmk}

In the above definition, it is to be expected that several different choices concerning the ``descended'' module are possible. The next lemma, which we state here for future use, explains how to be more canonical.

\begin{lem}\label{29.03.2018--1}Let $Y$ be proper and  $H^0$-flat over $A$. Write $B=H^0(\co_Y)$ and let $F$ be a finite $A$-module.  Then the canonical morphism  
\[
\co_Y\otu{B}H^0(\co_Y\ot_AF)\aro \co_Y\ot_AF
\]
is an isomorphism. 
\end{lem}

\begin{proof}By definition of $H^0$-flatness,  the canonical arrow 
\[\si:B\ot_AF\aro H^0(\co_Y\ot_AF)
\]
is bijective. We then consider the commutative diagram 
\[
\xymatrix{\co_Y\otu{B}H^0(\co_Y\ot_AF)\ar[rr]&&\co_Y\ot_AF\\\co_Y\otu{B}(B\ot_AF)\ar[rru]_\sim\ar[u]^-{\sim}_-{{\rm id}\ot\si}}
\] 
where the horizontal arrow is the one of the statement. 
\end{proof}

Here is the central definition of this section. 
\begin{dfn}\label{05.09.2018--3}If $\ph:Y\to X$ is a morphism of $A$-schemes, we let $\g T_\ph$   stand for the \emph{full} subcategory of $\bb{coh}(X)$ consisting of those coherent sheaves $\ce$ such that the coherent $\co_Y$-module $\ph^*(\ce)$ is trivial relatively to $A$. 
\end{dfn}

Let $X$ be a connected, proper and flat $A$-scheme, and   $x_0$ be an  $A$-point of $X$. Let  $\ph:Y\to X$ an object of $\g S(X,x_0)$ (see Definition \ref{26.06.2018--1}). 
If $B=H^0(\co_Y)$, we let $y_0$ stand for the $B$-point of $Y$ above $x_0$.
 We note that the $A$-module $B$ is finite and torsion-free, hence free. The following will be useful further ahead. 
 
\begin{lem}\label{21.02.2019--1} In the above notation, a coherent module $\cf\in\bb{coh}(X)$ belongs to $\g T_\ph$ if and only if $\ph^*(\ce)$ is trivial relatively to $B$. 
\end{lem}
\begin{proof}We only show the ``if'' clause. In this case, $\ph^*(\cf)\simeq\co_Y\ot_BF$ and using the point $y_0$, we conclude that $F\simeq \ph^*(\cf)|_{y_0}$. Now, since $y_0$ is taken to the $A$-point $x_0$, we have that $\ph^*(\cf)|_{y_0}\simeq B\ot_A(\cf|_{x_0})$ so that $\ph^*(\cf)\simeq \co_Y\ot_A(\cf|_{x_0})$.
\end{proof}


We now give ourselves an object $\ce$ of $\g T_\ph$: this means that $\ph^*(\ce)$ is isomorphic to the pull-back of a certain $A$-module to $Y$. 
In what follows we wish to find more propitious candidates for the latter isomorphism and module so to achieve a proof of Proposition \ref{05.04.2018--3} below. 
This Proposition is a key step in endowing the category $\g T_\ph$ with kernels and cokernels (see Section \ref{04.07.2018--1}). To have a better hold of the subtlety behind Proposition \ref{05.04.2018--3}, the reader is asked to read Remark \ref{01.02.2019--1} below.

As the category of finite $A$-modules is rather simple, we have: 
\begin{lem}\label{20.02.2019--4}
Let $E=\ce|_{x_0}$. Then the $\co_{Y}$-modules $\co_{Y}\ot_AE$ and $\ph^*(\ce)$ are (non-canonically) isomorphic.
\end{lem}
\begin{proof}We already know that $\ph^*(\ce)\simeq\co_{Y}\ot_A\ov E$ for {\it some} $A$-module $\ov E$, and all we are required to show is that $\ov E\simeq E$. Using the $B$-point $y_0$, we can say that $B\ot_A
\ov E\simeq B\ot_AE$ as $B$-modules.  Now, agreeing to write $A(\ell)=A/(\pi^\ell)$, we have 
\[
E\simeq A^r\op A(\de_1)\op\cdots\op A(\de_m)\quad\text{and}\quad \ov E\simeq A^{\ov r}\op A(\ov \de_1)\op\cdots\op A(\ov \de_{\ov m});
\]
here $m$ is either a positive integer, in which case $\de_1\ge\cdots\ge\de_m$ are also positive integers, or $m=0$ and the factors $A(\de_i)$ are to be dropped, and analogous considerations are in force for $\ov E$. 
As $B$ is free over $A$, of rank $s\ge1$ say, we have, as $A$-modules,  
\[
B\ot_AE\simeq A^{rs}\op \underbrace{A(\de_1)\op\cdots\op A(\de_1)}_{s}\op\cdots\op \underbrace{A(\de_m)\op\cdots\op A(\de_m)}_{s}
\]
and 
\[B\ot_A\ov E\simeq A^{\ov rs}\op \underbrace{A(\ov \de_1)\op\cdots\op A(\ov\de_1)}_{s}\op\cdots\op \underbrace{A(\ov \de_{\ov m})\op\cdots\op A(\ov \de_{\ov m})}_{s},
\]
so that the isomorphism $B\ot_AE\simeq B\ot_A\ov E$ implies that  $r=\ov r$, $m={\ov m}$ and $\de_i=\ov\de_i$. 
\end{proof}
From now on we write 
\[
\boxed{E:=\ce|_{x_0}}
\]  
and   let 
\[\boxed{
\tau:B\ot_AE\arou\sim H^0(\co_Y\ot_AE)
}\]
be the canonical isomorphism of $B$-modules. 

Of course one can find many isomorphisms $\co_Y\ot_AE\simeq\ph^*(\ce)$ and  we now  single out a special class of such. 
Using   $y_0$ and the fact that it ``extends'' $x_0$, 
we arrive at a canonical isomorphism 
\[
\ph^*(\ce)|_{y_0} \xymatrix{\ar[rr]^{\iota}_\sim&&}B\ot_AE
\] 
which, denoting by ${\rm ev}_{y_0}$ the ``evaluation'' $H^0(\ph^*(\ce))\to\ph^*(\ce)|_{y_0}$,   allows us to introduce the arrow of $B$-modules 
\[
\xymatrix{ H^0(\ph^*(\ce)) \ar[rr]^{\iota\ev_{y_0}}&& B\ot_AE.}
\]
Because $B\ot_AE\simeq H^0(\co_Y\ot_AE)$, 
the reader should note that  \emph{$\ev_{y_0}$ is an isomorphism}.
Hence, if $\al:\co_Y\ot_AE\to\ph^*(\ce)$ is an isomorphism, we arrive at a commutative diagram 
\begin{equation}\label{22.03.2018--1}
\xymatrix{H^0(\co_Y\ot_AE)\ar[rr]^-{H^0(\al)}&& H^0(\ph^*(\ce))\ar[d]^{\iota  {\rm ev}_{y_0}} \\ 
B\ot_AE\ar[u]^\tau\ar[rr] && B\ot_AE.
}
\end{equation}

\begin{dfn}\label{06.09.2018--1}\begin{enumerate}\item A global section $s\in H^0(\ph^*(\ce))$ is said to be \emph{conservative at $y_0$} if it is taken to $1\ot E$ under 
\[
\xymatrix{   H^0(\ph^*(\ce))        \ar[rr]^{\iota{\rm ev}_{y_0}}  && B\ot_AE. }
\] 

\item An isomorphism $\al:\co_Y\ot_AE\to\ph^*(\ce)$ is called {\it adapted to $y_0$} if the lower horizontal arrow in diagram \eqref{22.03.2018--1} is the identity. \end{enumerate}
\end{dfn}

\begin{rmk}Note that the canonical arrow $E\to 1\ot E$ is injective by faithful flatness of $A\to B$ and \cite[7.5(i), p.49]{matsumura}. 
\end{rmk}

\begin{lem}\label{23.03.2018--1}Isomorphisms adapted to $y_0$ always exist.
\end{lem}
\begin{proof}We first show that for an automorphism of $B$-modules 
\[
\xymatrix{H^0(\co_{Y}\ot_AE)\ar[rr]^c_\sim&&H^0(\co_{Y}\ot_AE)},
\]
it is   possible to find an automorphism of $\co_Y$-modules  
\[\xymatrix{\co_Y\ot_AE\ar[rr]^{\ga}_\sim&& \co_{Y}\ot_AE}
\]
such that $H^0(\ga)=c$. Write $\tilde c=\tau^{-1}c\tau$; this is an automorphism of $B\ot_AE$.  Pulling $\tilde c$ back to $Y$ via the structure morphism, we find an automorphism $\ga:\co_{Y}\ot_AE\stackrel\sim\to\co_{Y}\ot_AE$ such that $H^0(\ga)(\tau(1\ot e))=\tau(\tilde c(1\ot e))$. This implies that $H^0(\ga)(\tau(1\ot e))= c\tau(1\ot e)$. Since $\{\tau(1\ot e)\,:\,e\in E\}$ is a set of generators of $H^0(\co_{Y}\ot_AE)$, we see that $H^0(\ga)=c$. 

We now choose any isomorphism $\al:\co_Y\ot_AE\stackrel\sim\to\ph^*(\ce)$ and let $b$ be the composition 
\[
\xymatrix{H^0(\co_{Y}\ot_AE)\ar[rrr]^{\tau\circ \iota{\rm ev}_{y_0}\circ H^0(\al)}&&& H^0(\co_{Y}\ot_AE)}.
\]
Let $\be:\co_{Y}\ot_AE\stackrel\sim\to\co_{Y}\ot_AE$ induce $b^{-1}$ on global sections; it follows that $\al\be$ is adapted to $y_0$. 
\end{proof}

\begin{lem}\label{26.03.2018--1} If $\al:\co_{Y}\ot_AE\stackrel\sim\to\ph^*(\ce)$ is adapted  to $y_0$, then $H^0(\al)\tau$ defines a bijection between $1\ot_AE$ and the conservative sections of $H^0(\ph^*(\ce))$. 
\end{lem}
\begin{proof}
Because $\al$ is adapted to $y_0$, we see that $\iota\ev_{y_0}\left(H^0(\al)\tau(1\ot e)\right)=1\ot e$, which means that $H^0(\al)\tau$ sends $1\ot E$ into the conservative sections, so that  we only need to show that \emph{any} conservative section $s$ is of the form $H^0(\al)\tau(1\ot e)$. 
We write $s=H^0(\al)\tau(v)$ with $v\in B\ot_AE$. As $s$ is conservative $\iota\ev_{y_0}(s)= 1\ot e$ for some $e\in E$; as $\al$ is adapted, $\iota\ev_{y_0}\big(H^0(\al)\tau(v)\big)=v$, and hence $v=1\ot e$.  
\end{proof}

Granted these preparations, we can now have a better control on kernels and cokernels     in $\g T_\ph$. 

\begin{prp}\label{05.04.2018--3} Let $\ov \ce$ be another object of $\g T_\ph$ and $u:\ce\to\ov\ce$  a morphism of $\co_X$-modules.  Write $\ov E$ for the $A$-module $\ov\ce|_{x_0}$, $\ov\tau$ and $\ov\iota$ for the the canonical  morphisms $B\ot_A\ov E\to H^0(\co_Y\ot_A\ov E)$ and $\ph^*(\ov\ce)|_{y_0}\to B\ot_A\ov E$. Let  $\al:\co_{Y}\ot_AE\stackrel\sim\to\ph^*(\ce)$ and $\ov\al:\co_{Y}\ot_A\ov E\stackrel\sim\to\ph^*(\ov\ce)$ be adapted to the point $y_0$. The following claims hold true.

\begin{enumerate}\item The arrow $H^0(\ph^*u)$ takes conservative sections to conservative sections.

\item There exists a morphism of $A$-modules $u_0:E\to \ov E$ such that 
\[
\xymatrix{\ph^*\ce\ar[rr]^{\ph^*u}&& \ph^*\ov\ce 
\\
\ar[u]^\al\co_{Y}\ot_AE\ar[rr]_-{\id\otu A u_0}&&\co_{Y}\ot_A\ov E\ar[u]_{\ov\al}.
}\]
commutes. Put differently, $\ph^*(u)$ ``descends'' to $A$. 
\item The coherent sheaves  ${\rm Ker}(\ph^*u)$ and ${\rm Coker}(\ph^*u)$ are trivial relatively to $A$.
\end{enumerate}
\end{prp}

\begin{proof}
(1) 
Let $s\in H^0(\ph^*(\ce))$ be conservative at $y_0$, so that $\iota{\rm ev}_{y_0}(s)=1\ot e$ for some $e\in E$. Now 
\[
\ev_{y_0}\left\{H^0(\ph^*u)(s)\right\}=(\ph^*u)|_{y_0}\left\{\ev_{y_0}(s)\right\}.
\]
But, 
\[
\ov\iota\circ \left((\ph^*u)|_{y_0}\right)=\left(\id_B\ot_Au|_{x_0}\right)\circ\iota,
\]
which shows that $\ov\iota\ev_{y_0}\left\{H^0(\ph^*(u))(s)\right\}=\left(\id_B\ot_Au|_{x_0}\right)\circ\iota(\ev_{y_0}(s))=\left(\id_B\ot_Au|_{x_0}\right)(1\ot e)$.

(2) Let $e\in E$. We know that $s_e:=H^0(\al)\tau(1\ot e)$ is conservative at $y_0$ (Lemma \ref{26.03.2018--1}), so that  $
H^0\left( \ph^*(u)\right) (s_e)$ 
is equally conservative (by part (1)). Employing again Lemma \ref{26.03.2018--1}, we guarantee that  
\begin{equation}\label{05.09.2018--1}
H^0(\ph^*(u))(s_e)=H^0(\ov\al)\ov\tau(1\ot u_0(e))
\end{equation}
for some $u_0(e)\in \ov E$. Since  $\ov E\to B\ot_A\ov E$ is injective,       $u_0(e)$ is uniquely determined. This allows us to define a map $u_0:E\to \ov E$ which is easily seen to be $A$-linear. Now, 
\[ \xymatrix{\co_{Y}\ot_AE\ar[rr]^{\id\ot_Au_0}&&\co_{Y}\ot_A\ov E}\] 
 sends  $\tau(1\ot e)$ to $\ov\tau(1\ot u_0(e))$, and hence  
\[
\xymatrix{ \ph^*\ce\ar[rrr]^{\ov\al\circ (\id\ot_Au_0)\circ  \al^{-1}}&&& \ph^*\ov\ce} 
\]
sends $s_e=H^0(\al)\tau(1\ot e)$ to $H^0(\ov\al)\ov\tau(1\ot u_0(e))$. In conclusion, replacing $\ph^*(u)$ by $\ov\al\circ({\rm id}\ot_Au_0)\circ\al^{-1}$ in eq. \eqref{05.09.2018--1} still produces a true statement. Therefore $\ov\al\circ (\id\ot_Au_0)\circ  \al^{-1}=\ph^*(u)$.

(3) This follows from the fact that both kernel and cokernel of $\id\ot_Au_0$ are trivial relative to $A$ since $Y$ is flat over $A$.
\end{proof}

\begin{rmk}\label{01.02.2019--1}Let  $X=\spc {\pos\CC t}$ and $Y=\spc B$, where $B=\pos\CC {\sqrt t}$. Let $\ce=\co_X$ and define $v:\co_Y\to\co_Y$ as multiplication by $\sqrt t$. Clearly ${\rm Coker}(v)$ cannot be trivial relatively to $A$; of course this does not contradict  Proposition \ref{05.04.2018--3} since $v$ is not induced by any morphism $\co_X\to\co_X$.   
\end{rmk}

\section{Further properties of the category of coherent modules trivialized by a proper morphism}\label{04.07.2018--1}
In this section, we let $X$ be a proper and flat  $A$-scheme with {\it reduced fibres}, and $x_0$ an $A$-point of $X$. Let  $\ph:Y\to X$ an object of $\g S(X,x_0)$. (Recall that in such a situation $X$ is connected, see Remark \ref{05.02.2019--1}.) 
If $B=H^0(\co_Y)$, we let $y_0$ stand for the $B$-point of $Y$ above $x_0$. Let us gather some simple properties concerning the category $\g T_\ph$ of Definition \ref{05.09.2018--3}.


\begin{lem}\label{27.04.2017--3}Let $u:\ce\to\cf$ be arrow of $\g T_\ph$. Then  $\cC={\rm Coker}(u)$ belongs to $\g T_\ph$.  
\end{lem}

\begin{proof} As ${\rm Coker}(\ph^*u)\simeq \ph^*({\rm Coker}(u))$ this is a straightforward consequence of Proposition \ref{05.04.2018--3}-(3). 
\end{proof}

\begin{lem}\label{26.04.2017--1} If  $\ce\in\g T_\ph$ is $A$-flat, then $\ce$ is a vector bundle.  
\end{lem}
\begin{proof}By assumption, $X_k$ and $X_K$ are reduced schemes. It follows from \cite[Remarks (a),  p.226]{biswas-dos_santos11} that  $\ce_k$ and $\ce_K$ are \emph{vector bundles}. (We note that the context in op.cit. presupposes the ground field to be algebraically closed, but this is not necessary for the proof to work.) 
Using the ``fibre-by-fibre'' flatness criterion [EGA ${\rm IV}_3$, 11.3.10, p.138], we are done.
\end{proof}

We are now ready to state the main structure theorem concerning  $\g T_\ph$. 

\begin{thm}\label{06.04.2018--1}The category $\g T_\ph$ is abelian, and the inclusion functor $\g T_\ph\to\bb{Coh}(X)$ is exact. 
\end{thm}


The main point of the proof of Theorem \ref{06.04.2018--1}, given Lemma \ref{27.04.2017--3}, is to show that the kernel of an  arrow of $\g T_\ph$ is also in $\g T_\ph$ and, by Proposition \ref{05.04.2018--3}, all we need to do is to relate $\ph^*({\rm Ker}(u))$ and ${\rm Ker}(\ph^*(u))$ for an arrow $u$. 
 The argument takes up the ensuing lines.


First, recall that an $\co_X$-module $\cm$ is sait to be $\ph^*$-acyclic if the left derived functors $L_i\ph^*(\cm)$  vanish for $i>0$. (These functors are obtained by means of resolutions by flat $\co$-modules,  see \cite[99ff]{hartshorne66} . In addition, note that from \cite[Proposition 4.4, p.99]{hartshorne66}, if $\cm$ is  coherent, then $L_i\ph^*(\cm)$ is also  coherent.)

\begin{prp}\label{27.04.2017--1}  Any  $\ce$ in $\g T_\ph$  is $\ph^*$-acyclic.
\end{prp}

\begin{proof}
Let us first assume that $\pi\ce=0$ so that $\ce$ is by assumption an $\co_{X_k}$-module which becomes trivial after being pulled-back to $Y_k$. 
 Then, we know that $\ce$ is locally isomorphic to $\co_{X_k}^r$ \cite[Remarks (a), p.226]{biswas-dos_santos11}. 
Hence, there exists a faithfully flat morphism $\al:\ov X\to X$ (given by the disjoint union of open subsets of $X$) such that $\al^*(\ce)\simeq\co_{\ov X_k}^r$. 

Employing the notations introduced in the following cartesian diagram 
\[
\xymatrix{\ov Y\ar[r]^\be \ar[d]_{\ov\ph}\ar@{}[dr]|{\square}& Y\ar[d]^\ph \\ 
\ov X\ar[r]_\al&X}
\]
we now prove that $\co_{\ov X_k}^r$ is $\ov\ph^*$-acyclic. The exact sequence 
\begin{equation}\label{14.04.2018--2}
0\aro\co_{\ov X}^r\arou\pi\co_{\ov X}^r\aro\co_{\ov X_k}^r\aro0,
\end{equation}
gives rise to an exact sequence 
\[
0\aro L_1\ov\ph^*(\co^r_{\ov X_k})\aro    \ov\ph^*(\co_{\ov X}^r)\arou\pi \ov\ph^*(\co_{\ov X}^r) \aro \ov \ph^*(\co_{\ov X_k}^r)\aro 0,
\]
and this proves that  $L_1\ov \ph^*(\co_{\ov X_k}^r)=0$ as $\ov \ph^*(\co_{\ov X}^r)\simeq\co_{\ov Y}^r$ is $A$-flat. From sequence \eqref{14.04.2018--2}, we have, for any $i\ge1$, another exact sequence  
\[
0\,=\,L_{i+1} \ov \ph^*(\co_{\ov X}^r)\aro  L_{i+1}\ov \ph^*(\co_{\ov X_k}^r)\aro  L_{i}\ov \ph^*(\co_{\ov X}^r)\,=\,0;
\]
this allows us to conclude that $\co_{\ov X_k}^r$ is $\ov \ph^*$-acyclic. As  $\al^*$  and $\be^*$ are exact,  we know that $\be^* \circ  L_i\ph^*\simeq L_i\ov \ph^*\circ   \al^*$, and therefore we can say that, for each $i>0$, 
\[\begin{split}
\be^*\circ L_i\ph^*(\ce)&=L_i\ov \ph^*\circ \al^*(\ce)\\
&\simeq L_i\ov \ph^* (\co_{\ov X_k}^r)
\\&=0.
\end{split}
\] 
But $L_i\ph^*(\ce)$ is coherent and $\be$ is faithfully flat, so $\be^*( L_i\ph^*(\ce))=0$ implies $L_i\ph^*(\ce)=0$ and we have finished the verification that $\ce$ is $\ph^*$-acyclic.

For any $\cm\in\bb{Coh}(X)$, let us agree to write \[{\rm Tors}(\cm)=\bigcup_{m\ge1}{\rm Ker}\,\pi^m:\cm\to\cm\] for the sheaf of sections annihilated by some power of $\pi$ and define 
\[
t(\cm):=\min\{m\in\NN\,:\,\pi^m\po{\rm Tors}(\cm)=0\}. 
\] 
We shall show by induction on $t(\ce)$ that $\ce$ is $\ph^*$-acyclic. If $t(\cm)=0$, then  ${\rm Tors}(\cm)=0$, so that, due to Lemma \ref{26.04.2017--1}, $\ce$ is a vector bundle and a fortiori  $\ph^*$-acyclic.

Now, suppose that $t(\ce)\ge1$ and that  for all $\cf\in\g T_\ph$ with $t(\cf)<t(\ce)$, the $\co_X$-module $\cf$ is $\ph^*$-acyclic. 

Write $\ce':=\pi\ce$. Then, if $e'\in{\rm Tors}(\ce')(U)$ over some affine $U$, it follows that $e'=\pi e$ where $e\in{\rm Tors}(\ce)(U)$. This being so, we conclude that $\pi^{t(\ce)-1}e'=0$ and hence that $t(\ce')<t(\ce)$. Next, we consider the exact sequence 
\begin{equation}\label{05.04.2018--4a}0\aro\ce'\arou \rho\ce\arou \si\ce''\aro0,
\end{equation}
where $\ce''=\ce/\pi\ce$. Since, 
\[
\ce''={\rm Coker}\,\ce\arou\pi\ce,
\]
it follows that $\ce''\in\g T_\ph$ (there is no need to apply Lemma  \ref{27.04.2017--3} here). 
Note that, $\pi\ce''=0$ and hence $\ce''$ is $\ph^*$-acyclic by the first step in the proof. Consequently, we have the exact sequence 
\[
0\aro\ph^*(\ce')\aro\ph^*(\ce)\arou{\ph^*(\si)}\ph^*(\ce'')\aro0,
\]
which says that $\ph^*(\ce')={\rm Ker}(\ph^*(\si))$. But both $\ce$ and $\ce''$ are in $\g T_\ph$ so that \ref{05.04.2018--3}-(3) guarantees that $\ce'$ lies in $\g T_\ph$. As $t(\ce')<t(\ce)$, we can say that $\ce'$ is $\ph^*$-acyclic and hence, using the exact sequence \eqref{05.04.2018--4a} above, it follows that $\ce$ is equally $\ph^*$-acyclic. 
\end{proof}

We can now present our 

\begin{proof}[Proof of Theorem \ref{06.04.2018--1}]Let $u:\ce\to\cf$ be an arrow of $\g T_\ph$. If $\ck=\mm{Ker}(u)$, $\cC=\mathrm{Coker}(u)$ and $\ci={\rm Im}(u)$,  then we  have exact sequences 
\[0\aro\ck\aro\ce\aro\ci\aro0\quad\text{and}\quad0\aro\ci\aro\cf\aro\cC\aro0.\]
Note that $\cC\in{\g T_\ph}$ by Lemma \ref{27.04.2017--3}, so that Proposition  \ref{27.04.2017--1} tells us that  $\cC$ is $\ph^*$-acyclic.  This implies that the sequence 
\[0\aro \ph^*(\ci)\aro \ph^*(\cf)\aro \ph^*(\cC)\aro0\]
is exact and Proposition \ref{05.04.2018--3} assures that $\ph^*(\ci)$ is trivial relatively to $A$. Otherwise said, $\ci$ belongs to ${\g T_\ph}$. 
Proposition  \ref{27.04.2017--1} then proves that $\ci$ is $\ph^*$-acyclic.
Applying the same reasoning, we conclude that $\ck\in{\g T_\ph}$. 
\end{proof}

\begin{cor}\label{19.07.2018--2}The functor $\bullet|_{x_0}:{\g T_\ph}\to\modules A$ is exact and faithful. 
\end{cor}
\begin{proof}We know that $\bullet|_{x_0}$ is \emph{left} exact, so that we only need to show that $\bullet|_{x_0}$ preserves monomorphisms. 
Let $u:\ce\to\ov\ce$ be such a  monomorphism in ${\g T_\ph}$ and denote $\ce|_{x_0}$, respectively $\ov\ce|_{x_0}$, by $E$, respectively $\ov E$.  Let $\al:\co_Y\ot_AE\stackrel\sim\to\ph^*(\ce)$  and $\ov\al:\co_Y\ot_A \ov E \to\ph^*(\ov\ce)$ be adapted to $y_0$ (see Definition \ref{06.09.2018--1} and Lemma \ref{23.03.2018--1}).  Using Proposition \ref{05.04.2018--3},  we   see that $\ov\al^{-1}\ph^*(u)\al=\id_{\co_Y}\ot_Au_0$ for some $u_0:E \to\ov E$. Because $\ph^*u$ is a monomorphism (due to Lemma \ref{27.04.2017--3} and Proposition  \ref{27.04.2017--1}), we conclude that $u_0$ has to be a monomorphism: otherwise the pull-back functor $\modules A\to\bb{Coh}(Y)$ would fail to be exact and faithful. 
This implies that $\ph^*(u)|_{y_0}:\ph^*(\ce)|_{y_0}\to\ph^*(\ov\ce)|_{y_0}$ is a monomorphism, which shows that $B\ot_Au|_{x_0}$ is a monomorphism. 
Since $B$ is a faithfully flat $A$-algebra, we conclude that $u|_{x_0}$ is also a monomorphism. 

We now verify that $\bullet|_{x_0}$ is faithful by showing that if $\ce\in{\g T_\ph}$ is such that  $\ce|_{x_0}=0$, then $\ce=0$. Under this assumption, we see that $\ph^*(\ce)|_{y_0}=0$ which, together with an isomorphism $\co_Y\ot_AE\simeq\ph^*(\ce)$, proves the equality $\ph^*(\ce)=0$. Now, given any point $\xi\in X$, one easily sees using the surjectivity of $\ph$ that $\ce({\xi})=0$. Because of \cite[II, Exercise 5.8(c), p.125]{hartshorne77} and the fact that $X$ is reduced \cite[Corollary 23.9, p.184]{matsumura} we conclude  that  $\ce=0$. 
\end{proof}

In possession of Theorem \ref{06.04.2018--1} and Corollary \ref{19.07.2018--2}, we can now apply the theory of Tannakian categories over $A$ to define certain fundamental group schemes.

\begin{dfn}[The fundamental and Galois groups]\label{14.01.2019--2} 
\begin{enumerate}[i)] \item We shall let $\g T^\circ_\ph$ denote the full subcategory of $\g T_\ph$ whose objects are vector bundles. 
\item Given  $\ce\in\g T_\ph^\circ$, we denote by $\langle \ce\,;\,\g T_\ph\rangle_\ot$ the full subcategory of $\g T_\ph$ having as objects subquotients of   $\mathbf T^{a,b}\ce$ for varying multi-indexes $a,b$. 
\item Following \cite{hai-dos_santos18}, by $\g T_\ph^{\rm tan}$ we mean the full subcategory of $\g T_\ph$ whose objects are 
\[ 
\left\{\cv\in\g T_\ph\,\,:\,\,\text{$\cv$ is a quotient of some $\ce\in\g T_\ph^\circ$}\right\}=\bigcup_{\ce\in\g T^\circ_\ph}{\rm Ob}\,\langle\ce\,,\,\g T_\ph\rangle_\ot.
\]
\item  Given $\ce\in\g T_\ph^\circ$, we let ${\rm Gal}'(\ce,\g T_\ph,x_0)$, or simply ${\rm Gal}'(\ce)$ if context prevents any misunderstanding, stand for the   group scheme obtained from $\langle\ce\,;\,\g T_\ph\rangle_\ot$ via Tannakian duality (see \cite[Theorem 1.2.6]{duong-hai18} or \cite[II.4.1.1]{saavedra72}) by the functor $\bullet|_{x_0}:\langle\ce\,;\,\g T_\ph\rangle_\ot\to\modules A$.  
\item We let $\Pi(X,\ph,x_0)$ be the       flat group scheme obtained from $\g T_\ph^{\rm tan}$ via Tannakian duality by the functor $\bullet|_{x_0}: \g T^{\rm tan}_\ph\to\modules A$
\end{enumerate}
\end{dfn}
 
To end, we raise a point which was naturally suggested in the elaboration of the arguments in this section. (It should be compared to Exercise of \cite[Exercise 5.8, p.125]{hartshorne77}.)
\begin{question}Let $Z$ be a  flat, reduced and  noetherian $A$-scheme. Call a dvr-point of $Z$ any morphism $z:{\rm Spec}\,R\to Z$ of $A$-schemes where $R$ is a dvr. Now, for $\cf\in\bb{coh}(Z)$ and $z:{\rm Spec}\,R\to Z$ a dvr-point of $Z$, define its ``type'' as the invariants of $\cf|_z$. Assume now that the type of $\cf$ is the same for all $z$. Is it true that $\cf$ trivial relatively to $A$? 
\end{question}

\section{The generic fibre of the group scheme $\Pi$}\label{11.04.2019--4}

In this section, we let $X$ be a proper and flat  $A$-scheme with {\it reduced fibres}, and $x_0$ an $A$-point of $X$. Let  $\ph:Y\to X$ an object of $\g S(X,x_0)$.

Let $\g T_{\ph_K}$ stand for the full subcategory of $\bb{coh}(X_K)$ whose set of objects is 
\[
\left\{\begin{array}{c}E\in\bb{coh}(X_K)\,:\,\text{$\ph_K^*E$ is trivial}\end{array}\right\}.
\]
We note that $\g T_{\ph_K}$ is a Tannakian category over $K$ which is neutralized by the functor $\bullet|_{x_{0K}}:\g T_{\ph_K}\to\modules K$. (The reader is right to object that we have not proved this last claim since in Section \ref{04.07.2018--1} we work with a discrete valuation ring, but the translation is evident and in fact much simpler.)

\begin{thm}\label{11.04.2019--1}Suppose that $X_K$ is normal and let  $E\in\g T_{\ph_K}$. Then, $E$ is essentially finite and  $\langle E\,;\,\g T_{\ph_K}\rangle_\ot=\langle E\,;\,\bb{EF}(X_K)  \rangle_\ot$. In particular, the $K$-group scheme associated to $\langle E\,;\,\g T_{\ph_K}\rangle_\ot$ via the fibre functor $\bullet|_{x_{0K}}$ is finite. 
\end{thm}

\begin{proof}Let $\g C$ be the full subcategory of all vector bundles over $X_K$ which become trivial after being pulled back via some proper and surjective morphism. 
Then  \cite[Theorem 1]{antei-mehta11} or \cite[Theorem I]{tonini-zhang17} guarantee that   $\g C=\bb{EF}(X_K)$ and in particular $\g T_{\ph_K}$ is a full subcategory of $\bb{EF}(X_K)$. 
Following \cite[Corollary 2.3]{biswas-dos_santos11}, let us show that  $\langle E\,;\,\g C\rangle_\ot=\langle E\,;\,\g T_{\ph_K}\rangle_\ot$. Clearly, each object in $\langle E\,;\,\g T_{\ph_K}\rangle_\ot$ is also an object of $\langle E\,;\,\g C\rangle_\ot$.  Let then $T$ be some tensor power of $E$ and $q:T\to Q$ be a quotient morphism with $Q\in\g C$. Then, $\ph_K^*(q)$ is a quotient morphism from 
$\ph_K^*(T)\simeq \co_{Y_K}^{\op r}$  to  $\ph_K^*Q$. From Lemma \ref{24.05.2018--1}, we see that $\ph_K^*Q$ is trivial, which means that $Q\in\g T_{\ph_K}$.  Using duality of vector bundles, we then conclude that $\langle E\,;\,\g C\rangle_\ot=\langle E\,;\,\g T_{\ph_K}\rangle_\ot$. 
\end{proof}

The following result was employed in the proof of Theorem \ref{11.04.2019--1}. 

\begin{lem}\label{24.05.2018--1}Let $V$ be a proper scheme over a field and  $\co_V^{\op r}\to\cq$  a quotient morphisms to a vector bundle $\cq$ of rank $d$. Suppose that for a certain surjective and proper morphism $f:V'\to V$, the pull-back $f^*(\det\cq)$ is   trivial. Then $\cq$ is trivial. 
\end{lem}

\begin{proof}
We let ${\rm Grass}(r,d)$ be the 
Grassmann scheme as described in \cite[5.1.5(2), 110ff]{nitsure05} and denote by $\cu$ the universal quotient of $\co_{{\rm Grass}(r,d)}^{\op r}$. 
By construction of ${\rm Grass}(r,d)$, there exists a morphism $\ga:V\to{\rm Grass}(r,d)$ such that $\cq=\ga^*(\cu)$. Consequently, $f^*\ga^*(\det \cu)$ is trivial. 
Let $\de:{\rm Grass}(r,d)\to \PP^n$ be the determinant (Pl\"ucker) immersion so that $\det(\cu)=\de^*\co(1)$ \cite[p. 114]{nitsure05}. We then conclude that $f^*\ga^*\de^*\co(1)$ is trivial. Using Chow's lemma \ega{II}{}{5.6.1, p.106}, there is no loss of generality in supposing $V'$ to be projective. Now,  
Exercise 8.1.7(a) on p. 331 of \cite{liu02} says that the closed schematic image of $\de\ga f $ in $\PP^n$ is finite over the ground field.  Hence, the image of $\ga$ is a closed subset with finitely many points so that $\ga^*(\cu)$, which is $\cq$, is trivial. 
\end{proof}

For $E\in\g T_{\ph_K}$, let us denote the group scheme over $K$ associated to $\langle E\,;\,\g T_{\ph_K}\rangle_\ot$ via the functor $\bullet|_{x_{0K}}$ by $\mm{Gal}(E,\g T_{\ph_K},x_{0,K})$. 

\begin{thm}\label{11.04.2019--2}Suppose that $X_K$ is normal and let $\ce\in\g T_\ph^\circ$. If $E$ stands for $\ce_K$, then $\mm{Gal}'(\ce,\g T_\ph,x_0)\ot K\simeq \mm{Gal}(E,\g T_{\ph_K},x_{0K})$. In particular, $\mm{Gal}'(\ce,\g T_\ph,x_0)\ot K$ is finite. 
\end{thm}
\begin{proof}

Let $\si:\langle\ce\,;\,\g T_\ph\rangle_\ot\to \langle E\,;\,\g T_{\ph_K}\rangle_\ot$ be the base change functor,  $G$ be the group scheme $\mm{Gal}'(\ce,\g T_\ph,x_0)$ and $H$ be       $\mm{Gal}(E,\g T_{\ph_K},x_{0K})$. 
 Let $\te:\rep AG\to\langle\ce\,;\,\g T_\ph\rangle_\ot$ be an inverse to $\bullet|_{x_0}:\langle\ce;\g T_\ph\rangle_\ot\to\rep AG$ and denote by $\tau:\rep AG\to\rep KH$ the composition
\[
\xymatrix{
\rep AG\ar[r]^\te &  \langle\ce\,;\,\g T_\ph\rangle_\ot  \ar[r]^-{\si} &  \langle E\,;\,\g T_{\ph_K}\rangle_\ot \ar[r]^-{\bullet|_{x_{0K}}} &\rep KH. }
\]

If $V\subset A[G]_{\rm right}$ contains the constants, the fact that $\te$ is fully faithful assures that  $H^0(X,\te(V))\simeq A$. Hence,  $H^0(X_K,\si\te(V))\simeq K$, by flat base-change. Consequently,  
\[
\tau(V)^H\simeq K.
\]

Let $i:\rep AG\to\rep K{G_K}$ be the extension of scalars. It is not difficult to see that there exists a morphism $\xi:H\to G_K$ of group schemes such that 
\[
\xymatrix{\rep A{G} \ar[d]_i \ar[r]^\tau &  \rep K{H}
\\
\rep K{G_K}\ar[ru]_{\xi^\#},& 
}
\]
is commutative (up to natural isomorphism of tensor functors).

If $V\subset K[G]_{\rm right}$, there exists $V^\flat\subset A[G]_{\rm right}$ containing the constants and an injection  $V\to i(V^\flat)$. Then, 
\[\begin{split}\xi^\#(V)^{H}&\subset \tau(V^\flat)^{H}\\&\simeq K.
\end{split}
\]
This implies that $K\simeq K[G]^H_{\rm right}$, and we conclude, with the help of Lemma \ref{29.11.2018--1} and the fact that $H$ is finite (Theorem \ref{11.04.2019--1}), that $\xi$ is faithfully flat. On the other hand, since $E|_{x_{0K}}$ is a faithful representation of $H$, the standard criterion \cite[2.21, p.139]{deligne-milne82} immediately shows that $\xi$ is a closed immersion: it follows that $\xi$ is an isomorphism. 
\end{proof}

As a simple consequence we have 
\begin{cor}\label{16.04.2019--1}Suppose that $X_K$ is normal. Then $\Pi(X,x_0)\ot K$ is pro-finite. \qed 
\end{cor}

\begin{rmk}\label{11.04.2019--3}We are unable to show that Theorem \ref{11.04.2019--2} holds without the finiteness assumption on $\mm{Gal}(E,\g T_{\ph_K},x_0)$. This is because, once this assumption is removed, Lemma \ref{29.11.2018--1} cannot be applied so that, in its place, we need  the standard criterion guaranteeing faithful flatness \cite[2.21, p.139]{deligne-milne82}.
\end{rmk}

\section{Prudence in the category $\g T_*$}\label{14.03.2019--1}
We shall assume that   $A$ is complete.
In \cite[Section 6]{hai-dos_santos18} we   introduced the notion of {\it prudence} of an affine flat group scheme over $A$. As shown in op. cit. this concept turns out to be equivalent to  the Mittag-Leffler property (or freeness)  of the $A$-module of functions of the group. In this section we show that the $A$-module of functions of a Galois group (see Definition \ref{14.01.2019--2}) is free by using prudence. 

Here are the hypothesis on the ambient space which are fixed in this section. Let  $X$ be a proper and flat $A$-scheme having reduced fibres. We give ourselves an $A$-point $x_0$ and an arrow $\ph:Y\to X$ in $\g S^+(X,x_0)$ (see Definition \ref{26.06.2018--1}). If $B=H^0(\co_Y)$, we let $y_0$ be the $B$-point of $Y$ above $x_0$ assured by the definition of $\g S^+$.

\begin{thm}\label{25.06.2018--4}Let  $\ce\in\g T_\ph^\circ$. Then ${\rm Gal}'(\ce,\g T_\ph,x_0)$, as introduced in Definition \ref{14.01.2019--2}, is prudent. In particular, its $A$-module of functions is free. 
\end{thm}

Before starting the proof, 
let us obtain the

\begin{cor}\label{25.06.2018--2}Suppose that $X_K$ is normal. Then ${\rm Gal}'(\ce,\g T_\ph,x_0)$ is finite.   
\end{cor}

\begin{proof}We abbreviate $G:={\rm Gal}'(\ce,\g T_\ph,x_0)$. By  Theorem \ref{11.04.2019--2},   $G\ot K$ is finite over $K$. As $A$ is complete, the $A$-module $A[G]$ is isomorphic to a sum $A^{a}\op K^b$ by \cite[Theorem 12]{kaplansky52}. But Theorem \ref{25.06.2018--4} implies that $b=0$, so $A[G]\simeq A^a$.
\end{proof}

\begin{proof}[Proof of Theorem \ref{25.06.2018--4}]    Recall that by assumption   $Y$ is $H^0$-flat over $A$ and 
flat over $B=H^0(\co_Y)$. We wish to show that ${\rm Gal}'(\ce)$ is prudent \cite[Section 6]{hai-dos_santos18}, and for that we introduce the following setting. 


Let $\cv\in\langle\ce\,;\,\g T_\ph\rangle_\ot$ be a vector bundle and let \[v_n:\cl_n\aro\cv_n\] be a compatible family of 
monomorphisms of $\co_X$-modules 
(as customary, $\cv_n=A_n\ot_A\cv$, $B_n=A_n\ot_AB$, etc.) such that for each $n\in\NN$,  the $\co_X$-module $\cl_n$ belongs to $\g T_{\ph}$. By Grothendieck's existence theorem \cite[Theorem 8.4.2]{illusie06}, there exists an arrow of coherent $\co_X$-modules $v:\cl\to\cv$ such that $v_n$ is none other than $v\ot_AA_n$.

Write $\cm:=\ph^*(\cl)$; our assumption on $\cl_n$ tells us that $\cm_n=\ph^*(\cl_n)$ is trivial relatively to $A$. Then, since we assume that $Y$ is $H^0$-flat over $A$, Lemma \ref{29.03.2018--1} tells us that   the canonical arrow 
\[\xymatrix{
\co_Y\otu BH^0(\cm_n)\ar[rr]^-{\te_n}&&    \cm_n}
\]
is an \emph{isomorphism}.
We now require the 

\begin{lem}\label{14.04.2018--1}Under the above notations and assumption, the canonical morphism of $\co_Y$-modules 
\[\tau:\co_Y\otu BH^0(\cm)\aro\cm\]
is an isomorphism. 
In particular, $\cm$ is trivial relatively to $B$. 
\end{lem}

\begin{proof}
Consider the commutative diagram 
\[
\xymatrix{
\cm_{n+1}\ar[rr]^{\pi^{n+1}} && \cm_{n+1}\ar[rr]^{q_n}&& \cm_n\ar[r]&0
\\
\ar[u]^{\te_{n+1}}\co_Y\otu BH^0(\cm_{n+1})\ar[rr]_{\pi^{n+1}}  && \co_Y\otu BH^0(\cm_{n+1}) \ar[rr]_{\id\ot H^0(q_n)}\ar[u]^{\te_{n+1}}&& \co_Y\otu BH^0(\cm_{n})\ar[u]^{\te_{n}}\ar[r]&0,
}
\]
where $q_n$ is the canonical arrow. 
By assumption, the vertical arrows are isomorphisms while the upper row is tautologically an exact sequence of $\co_Y$-modules. We conclude that the bottom row is exact and faithful flatness of $Y$ over $B$ (faithfulness is guaranteed by \ega{III}{1}{4.3.1, p.130}) shows that the sequence 
\[
\xymatrix{
H^0(\cm_{n+1})\ar[r]^{\pi^{n+1}} & H^0(\cm_{n+1})\ar[rr]^{H^0(q_n)}& & H^0(\cm_n)\ar[r]&0
}
\] is equally exact. Hence, if 
\[\g M:=\lip\left( \xymatrix{\cdots\ar[r]&H^0(\cm_{n+1})\ar[rr]^{H^0(q_n)}& & H^0(\cm_{n})\ar[r]&\cdots} \right),\]
it follows that $\g M$ is a finitely generated $B$-module and that the projection 
\[
\g M\aro H^0(\cm_n)
\] 
induces an isomorphism 
\[
B_{n}\ot_B\g M\arou\sim H^0(\cm_n). 
\] 
(This follows from the the fact that $B$ is $\pi$-adically complete and \ega{0}{\rm I}{7.2.9}.) More importantly, letting    $u_n:\cm\to\cm_n$ stand for  the natural epimorphism, a direct application of the theorem of formal functions (see \cite[8.2.4, p. 188]{illusie06} or \cite[III.11.1, p.277]{hartshorne77}) guarantees that the obvious arrow
\[  
\xymatrix{ H^0(\cm) \ar[rrr]^-{\lip_nH^0(u_n)} &&& \g M }
\]
is a bijection. Consequently, for any given $n\in\NN$, the arrow $H^0(u_n): H^0(\cm)\to H^0(\cm_n)$ 
induces an isomorphism 
\[ 
U_n: B_{n}\otu BH^0(\cm)\arou\sim H^0(\cm_n).
\]
To show that  $\tau:\co_Y\ot_BH^0(\cm)\to\cm$ is an isomorphism, by Grothendieck's algebraization theorem \cite[Theorem 8.2.9, p.192]{illusie06}, we only need to show that for each $n$, 
\[
B_n\otu B (\co_Y\ot_BH^0(\cm))\xymatrix{\ar[rr]^-{\id\ot \tau}&&} B_{n}\otu B \cm
\]
is an isomorphism. Now, we have the commutative diagram 
\[\xymatrix{
B_{n}\otu B  (\co_Y\ot_BH^0(\cm))\ar[rr]^-{\id\ot\tau}&&B_{n}\otu B \cm \ar[r]^{\ov u_n}_\sim & \cm_n \\ 
\ar[u]^-{\sim}_{\text{canonic}} \co_Y\otu B(B_{n}\ot_B H^0(\cm))\ar[rrr]_-{\id \ot U_n} &&& \co_Y\otu BH^0(\cm_n)\ar[u]^-{\te_n}_-{\sim},
}
\]where $\ov u_n$ is associated to $u_n$, and the conclusion follows.  
\end{proof}

By virtue of the above Lemma, $\cm$ is trivial relatively to $B$; using Lemma \ref{21.02.2019--1} we conclude that $\cm$ is trivial relatively to $A$.
This shows that 
$\cl\in\g T_\ph$. As it is not difficult to show that $v$ is   a monomorphism by looking at the arrow $\cl|_{x_0}\to\cv|_{x_0}$, we see that $\cl$ does indeed belong to $\langle\ce\,,\,\g T_\ph\rangle_\ot$.

We have then showed that the group scheme ${\rm Gal}'(\ce,\g T_\ph)$ is prudent, and hence  that its ring of functions is free as an $A$-module \cite[Section 6]{hai-dos_santos18}. 

\end{proof}



\section{Digression on the reduced fibre theorem}\label{RFT}
For future applications, we shall state a version of the reduced fibre theorem (see \cite[Theorem $2.1'$]{bosch-lutkebohmert-raynaud95} and \cite{temkin10}) suitable for us. 
\begin{thm}[The reduced fibre theorem]\label{reduced_fibre_theorem} Let $A$ be  Henselian and $Y$ be a flat $A$-scheme of finite type whose fibre $Y_K$ is geometrically reduced. Then, there exists a finite extension of discrete valuation rings $B\supset A$ and a commutative diagram 
\[
\xymatrix{&Z\ar[d]^\ep
\\
Y\ar[d]&\ar[l]\ar@{}|{\square}[dl] Y_B\ar[d]
\\
\spc A &\ar[l] \spc B
}
\]
where: 
\begin{enumerate}[{\bf RFT}1)]\item The $B$-scheme $Z$ is flat   and has geometrically reduced fibres. 
\item The morphism $\ep$ is finite and surjective. 
\item If $L={\rm Frac}(B)$, then 
\[
\ep\otu BL:Z\otu BL\aro Y_B\otu BL
\] 
is an isomorphism. 
\end{enumerate}In particular, the composition $Z\to Y_B\to Y$ is finite and surjective. 
\end{thm}

\begin{proof}We show how to arrive at the conclusion starting from \cite[Theorem 3.5.5]{temkin10}. 
By this theorem (see also the notation on p. 619 of op. cit.), we obtain the existence of an integral scheme $S$ and a separable alteration (proper, dominant and inducing a finite and separable extension of function fields) 
\[
S\aro \spc A,
\]
and a commutative diagram with cartesian square 
\[
\xymatrix{
&\wt Y\ar[d]^\de
\\
Y\ar[d]&Y_S \ar[d]\ar[l]\ar@{}[ld]|\square
\\
\spc A& \ar[l]S,
}
\]
such that: 
\begin{itemize}
\item The $S$-scheme $\wt Y$ is flat and has geometrically reduced fibres. 
\item The morphism $\de$ is finite.
\item  The base-change 
\[
\xymatrix{\wt{Y}\tiu SS_K\ar[rr]^-{\de\tiu S S_K}&& Y_S\tiu SS_K} 
\]
is an isomorphism.
\end{itemize}
As $Y_S\to S$ is an open morphism  \ega{IV}{2}{2.4.6} and $S_K$ is dense in $S$, we can say that $Y_S\ti_SS_K$ is dense in $Y_S$; consequently, 
\begin{itemize}
\item The morphism $\de$ is \emph{surjective}.
\end{itemize}
If $L$ is the function field of $S$ and $B$ is the integral closure of $A$ in $L$ (necessarily a d.v.r. and a finite $A$-module \cite[31.B, p.232]{matsumuraCA}), properness of $S$ gives the existence of a morphism of $A$-schemes $\spc B\to S$ extending $\spc L\to S$.  
We then arrive at a commutative diagram with cartesian squares 
\[\xymatrix{
&\wt{Y}\ar[d]^\de & Z\ar[d]^\ep\ar[l]\ar@{}|\square[dl]
\\
Y\ar[d] & Y_S\ar@{}|\square[dr] \ar[d]\ar[l]\ar@{}[ld]|\square & Y_B\ar[d]\ar[l]
\\
\spc A& \ar[l]S & \spc B\ar[l]
}
\]
where  
\begin{itemize}\item The $B$-scheme $Z$ is flat and has geometrically reduced fibres (for the last condition, see  \ega{IV}{2}{4.6.1, p.68}). 
\item The morphism $\ep$ is finite and surjective (for surjectivity, see \ega{I}{}{3.5.2, p.115}). 
\item The base-change morphism  
\[
\ep\otu BL:Z\otu BL\aro Y_B\otu BL
\] 
is an isomorphism. 

\end{itemize}
The proof of the last claim is quite simple and we omit it. 
\end{proof}

The ensuing result will prove useful in order to apply Theorem \ref{reduced_fibre_theorem}. In it, we employ the notion of Japanese discrete valuation ring \ega{0}{{\rm IV}}{23.1.1, p.213}. (The reader should recall that all complete d.v.r.'s are Japanese, as are those whose field of fractions has characteristic zero. See \ega{0}{{\rm IV}}{23.1.5} and \ega{0}{{\rm IV}}{23.1.2}.)

\begin{lem}\label{25.10.2018--1}
Let $A$ be a Japanese   and $S$ a flat $A$-scheme of finite type. There exists a finite purely inseparable extension $\tilde K$ of $K$ such that the following property holds. If $\tilde A$ is the integral closure of $A$ in $\tilde K$, then   the  $\tilde A$-scheme $\tilde S:=(S\ot_A\tilde A)_\red$  has a geometrically reduced generic fibre and is flat . 
\end{lem}
\begin{proof}

Let $\tilde K/K$ be a finite and purely inseparable extension such that $(S\ot_A\tilde K)_\red$ is   geometrically reduced over $\tilde K$ \ega{IV}{2}{4.6.6, p.69}. Now, in the notation of the statement, $\tilde S[1/\pi]=(S\ot_A\tilde K)_\red$, and we are done. To prove that $\tilde S$ is flat over $\tilde A$, let $\tilde \pi$ be a uniformizer of $\tilde A$.
 Let $\tilde f$ be a function on some unspecified affine open subset of $S\ot_A\tilde A$ such that   $\tilde \pi \tilde f$ is nilpotent,  say $\tilde \pi^m \tilde f^{m}=0$. Then,   the fact that $S\ot_A\tilde A$ is $\tilde A$-flat says that $\tilde f$ is nilpotent, and this shows that $(S\ot_A\tilde A)_\red$ is $\tilde A$-flat.
\end{proof}

\section{The fundamental group scheme}\label{FGS}

 Let $X$ be an irreducible, proper and flat $A$-scheme with geometrically reduced fibres, and $x_0$ an $A$-point of $X$. We now wish to assemble the categories $\g T_\ph$ (see Definition \ref{05.09.2018--3} and Theorem \ref{06.04.2018--1}) for varying $\ph$ in a single one, and for that we need the following:

\begin{thm}\label{05.11.2018--1}We suppose that $A$ is Henselian and Japanese.   
Let 
$\ph:Y\to X$ be a proper and surjective morphism. Then, there exists a commutative diagram of schemes 
\[
\xymatrix{Z\ar[r] \ar[dr]_\ps&Y\ar[d]^\ph\\&X}
\] 
enjoying the following properties. 
\begin{enumerate}[(1)]\item The morphism $\ps$ is surjective and proper. 
\item The ring $C:=H^0(\co_Z)$ is a discrete valuation ring and is a finite extension of $A$.
\item The canonical morphism $Z\to\spc C$ is flat and has  geometrically reduced fibres. 
\item The $C$-scheme  $Z$  has a $C$-point above $x_0$. 
\end{enumerate}
In particular, $\ps:Z\to X$ belongs to the category  $\g S^+(X,x_0)$ introduced in Definition \ref{26.06.2018--1}. 
\end{thm}

\begin{proof}The construction requires several steps.  

\vspace{.2cm}
\noindent\emph{Step 1.} Because $X$ is {\it irreducible}, there exists an irreducible component $Y'$ of $Y$ such that $\ph:Y'\to X$ is         surjective. 

\vspace{.2cm}
\noindent\emph{Step 2.}
Let $j:Y''\to Y$ be   reduced closed subscheme underlying $Y'$. Note that $Y''$ is integral and that $\ph'':=\ph \circ j:Y''\to X$ is proper and surjective.   

\vspace{.2cm}
\noindent\emph{Step 3.}
Let now $\nu:Y^\dagger\to Y''$ be the normalization \ega{II}{}{p.119}. Since $A$ is universally Japanese \ega{IV}{2}{7.7.2, p.212}, we conclude that $\nu$ is finite, and hence $\ph^\dagger:=\ph''\circ\nu$ is proper and surjective. In addition, $H^0(\co_{Y^\dagger})$ is a normal integral domain (see \cite[2.4.17, p.65]{liu02} and \cite[4.1.5, p.116]{liu02}). 
The morphism $Y^\dagger \to\spc A$ is surjective  and hence flat \cite[4.3.10,p.137]{liu02}; it then follows that $H^0(\co_{Y^\dagger})$ is a finite and flat $A$-module. As such, it must in addition be a local ring, since $A$ is Henselian.  Consequently,  $H^0(\co_{Y^\dagger})$ is a Dedekind domain, and  a fortiori a discrete valuation ring.  In hindsight, the canonical morphism $Y^\dagger\to\spc H^0(\co_{Y^\dagger})$ is surjective and flat \cite[4.3.10,p.137]{liu02}.

\vspace{.2cm}
\noindent\emph{Step 4.}
Let $B^\dagger:=H^0(\co_{Y^\dagger})$ and write $L^\dagger$ for its field of fractions. Let $\tilde L$ be a purely inseparable extension of $L^\dagger$   such that, denoting by $\tilde B$ the integral closure of $B^\dagger$ in $\tilde L$, the generic fibre of   
\[
 (Y^\dagger\ot_{B^\dagger}\tilde B)_\red\aro \spc {\tilde B} 
\]
is geometrically reduced (we employ Lemma \ref{25.10.2018--1}). Writing $
\tilde Y:=(Y^\dagger\ot_{B^\dagger}\tilde B)_\red$, 
we note that  the composition 
\[
\tilde Y\aro {Y^\dagger\ot_{B^\dagger}\tilde B}\arou{\mm{pr}} Y^\dagger
\] 
induces a homeomorphism between the underlying topological spaces. From this, it follows that $\tilde Y$ is  integral and that the obvious morphism $\tilde\ph:\tilde Y\to X$ is surjective. Since $\tilde B$ is a finite $B^\dagger$-module---recall that $A$ is Japanese---, we conclude that       $\tilde\ph$ is proper.  
Since $\tilde B$ is a Dedekind domain,  the fact that  $\tilde Y\to\spc \tilde B$ is surjective (note that $Y^\dagger\ot_{B^\dagger}\tilde B\to\spc\tilde B$ is surjective by \ega{I}{}{3.5.2, p.115}) shows that it is in addition flat.

For future usage, we also remark that 
\begin{align*}H^0(\co_{\tilde Y})&=H^0\left(\co_{Y^\dagger\ot_{B^\dagger}\tilde B}\right)_\red & \text{(see  \cite[2.4.2(c),p.60]{liu02})}
\\
&=\left(H^0(\co_{Y^\dagger})\ot_{B^\dagger}\tilde B\right)_\red &\text{(flat base-change \cite[3.1.24. p.85]{liu02})} 
\\
&=\left(B^\dagger\ot_{B^\dagger}\tilde B\right)_\red& \text{(definition of $B^\dagger$)}
\\
&=\tilde B.&
\end{align*}

\vspace{.2cm}
\noindent\emph{Step 5.} 
By the reduced fibre theorem (Theorem \ref{reduced_fibre_theorem}), 
there exists a finite extension $B^\natural\supset\tilde B$ of discrete valuation rings  and a commutative diagram 
\[
\xymatrix{&Y^\natural\ar[d]^\ep
\\
\tilde Y\ar[d] & \tilde Y\otu{\tilde B}B^\natural \ar[d]\ar[l]
\\
\spc \tilde B&\ar[l]\spc B^\natural}
\]
such that 
\begin{itemize}\item The $B^\natural$-scheme $Y^\natural$ is flat, proper and has geometrically reduced fibres. 
\item The morphism $\ep$ is finite and surjective. 
\item If $L^\natural={\rm Frac}(B^\natural)$, then 
\[
\ep\otu{B^\natural} L^\natural: Y^\natural\otu{B^\natural}L^\natural\aro(\tilde Y\ot_{\tilde B}B^\natural)\otu{B^\natural}L^\natural
\]
is an isomorphism. 
\end{itemize}
Furthermore, 
\begin{align*}
H^0\left(\co_{Y^\natural\otu{B^\natural}L^\natural}\right)&\simeq H^0(\co_{\tilde Y\ot_{\tilde B}L^\natural})&\text{(via $\ep\ot_{B^\natural}L^\natural$)} 
\\
&\simeq H^0(\co_{\tilde Y})\otu{\tilde B}L^\natural&\text{(by flat base-change)}
\\
&\simeq L^\natural&\text{(by Step 4)}.
\end{align*}
Let $\ph^\natural:Y^\natural\to X$ be the composition 
\[
Y^\natural\arou\ep \tilde Y\otu{\tilde B}B^\natural\arou{\rm pr} \tilde Y\arou{\tilde\ph}X;
\]
as $\tilde\ph$ is surjective and proper, it is clear that $\ph^\natural$ is surjective and proper. 

\vspace{.2cm}
\noindent\emph{Step 6.} 
Let now $C\supset B^\natural$ be a finite extension of d.v.r.'s such that $Y^\natural$ has a $C$-point above the $A$-point $x_{0}$ (here we apply the valuative criterion of properness for the $A$-scheme $Y^\natural|_{x_0}$ to obtain a finite extension $A\subset C$). We then define $Z=Y^\natural\ot_{B^\natural}C$ and note that the following claims hold true:  
\begin{itemize}\item The composition   $Z\stackrel{\rm pr}\to Y^\natural\stackrel{\ph^\natural}\to X$
is surjective and proper.
\item  The $C$-scheme $Z$ is flat, proper and has geometrically reduced fibres. 

\item Let $M=\mm{Frac}(C)$. Then, $H^0(\co_{Z\ot_CM})=M$ (because of $H^0(\co_{Y^\natural\ot_{B^\natural}L^\natural})=L^\natural$ and flat base-change).
\item The ring of global functions of $Z$ is $C$  (because of the previous claims and Lemma \ref{03.07.2018--2}.)
\item There is a $C$-point in $Z$ above the $A$-point $x_0$. 
\end{itemize}
This proves all conclusions in the statement. 
\end{proof}

\begin{cor}\label{19.07.2018--1}We suppose that $A$ is Henselian and Japanese.   
  Let $\ph:Y\to X$ and $\ph':Y'\to X$ be surjective and proper morphisms. Then there exists $\ps:Z\to X$ in $\g S^+(X,x_0)$ and a commutative diagram
\[
\xymatrix{&Z\ar[dd]^\ps\ar[dr]\ar[dl]&
\\
Y\ar[dr]_\ph&&Y'\ar[dl]^{\ph'}
\\
&X.&
}
\] 
In addition, it is possible to find $Z$ such that the  extra conditions hold:
  \begin{enumerate}\item The ring  $H^0(\co_Z)$ is a d.v.r. which is a finite extension of $A$,  and 
  \item the canonical morphism $Z\to\spc H^0(\co_Z)$ is flat with geometrically reduced fibres. \end{enumerate}\qed
\end{cor}


\begin{dfn}\label{30.10.2018--1}We let $\g T_X$, stand for the full subcategory of $\bb{coh}(X)$ 
whose objects are 
\[
\bigcup_{\ph\in\g S(X,x_0)}{\rm Ob}\, {\g T}^{\rm tan}_\ph.
\]In addition, $\g T_X^\circ$ is the full subcategory of $\g T_X$ having objects which are in addition vector bundles. 
\end{dfn}

In the terminology of the definition we can say: 

\begin{cor} If $A$ is Henselian and Japanese, then the category $\g T_X$ is a full abelian subcategory of $\bb{coh}(X)$ which is in addition stable under tensor products. With this structure and with the functor   $\bullet|_{x_0}:\g T_X\to\modules A$, $\g T_X$ becomes a neutral  Tannakian category over $A$ in the sense of \cite[Definition 1.2.5, p.1109]{duong-hai18}.  
\end{cor}

\begin{proof}
Let $\ph:Y\to X$ and $\ph':Y'\to X$ belong to $\g S(X,x_0)$ and $\ps:Z\to X$ be as in Corollary \ref{19.07.2018--1}. Then, $\g T^{\rm tan}_\ph$ and $\g T^{\rm tan}_{\ph'}$ are full subcategories of $\g T^{\rm tan}_\ps$ and all the claims made in the statement follow from  Theorem \ref{06.04.2018--1} and Corollary \ref{19.07.2018--2}.
\end{proof}

\begin{dfn}\label{19.07.2018--3}The fundamental group scheme of $X$ at the point $x_0$ is the affine and flat group scheme obtained from $\g T_X$ and the functor $\bullet|_{x_0}:\g T_X\to\modules A$ via Tannakian duality \cite[Theorem 1.2.6]{duong-hai18}.  It shall be denoted by $\Pi (X,x_0)$. 

If $\ce\in\g T_X^\circ$, we let $\mm{Gal}'(\ce,\g T_X,x_0)$ be the flat group scheme over $A$ defined by the category $\langle\ce\rangle_\ot$ and the functor $\bullet|_{x_0}$. 
\end{dfn}

The next result clarifies the relation between 
 $\g T_X$ and its constituents $\g T_\ph$. 

\begin{prp}\label{07.11.2018--1}
Let $\ph:Y\to X$ be an arrow of $\g S(X,x_0)$ and $\ce\in\g T_\ph^\circ$.   
Then, the natural morphism 
$\nu:\Pi(X,x_0)\to\Pi(X,\ph,x_0)$ is faithfully flat, while 
${\rm Gal}'(\ce,\g T_X,x_0)\to{\rm Gal}'(\ce,\g T_\ph,x_0)$   is an isomorphism.  
\end{prp}
\begin{proof}Let us write $G:=\Pi(X,x_0)$ and $G^\ph:=\Pi(X,\ph,x_0)$. In addition, given any $A$-linear category $\g C$, we denote by $\g C_{(k)}$ the full subcategory of $\g C$ whose objects are ``annihilated by $\pi$'', meaning  that multiplication by $\pi$ coincides with 0. 
We shall prove that 
\begin{enumerate}[i)]\item For any $V\in\repo A{G^\ph}$ and any quotient morphism $q:\nu^\#(V)\to Q$ with $Q\in\repo AG$, there exists $Q^\ph\in\repo A{G^\ph}$ such that $\nu^\#(Q^\ph)=Q$.
\item For any $M\in \rep A{G^\ph}_{(k)}$ and any quotient morphism $q:\nu^\#(M)\to Q$ with $Q\in\rep AG_{(k)}$, there exists $Q^\ph\in\rep A{G^\ph}$ such that $\nu^\#(Q^\ph)=Q$.
\end{enumerate}
If these two conditions are verified, then the evident morphisms $A[G^\ph]\to A[G]$ and $k[G^\ph]\to k[G]$ are injective (as follows from the ``dual statements'' in  \cite[3.2.1]{duong-hai18} and \cite[2.21, p.139]{deligne-milne82}) and Theorem 4.1.1 in \cite{duong-hai18} proves that $\nu$ is faithfully flat.

To verify (i), we give ourselves a morphism  $\ps:Z\to X$ in $\g S(X,x_0)$,  an object $\cq$ from $\g T_{\ps}^\circ$, an object $\cv$ of $\g T_\ph^\circ$ and an epimorphism $q:\cv\to\cq$, and aim at showing that $\cq$ lies in $\g T_\ph^\circ$. By Lemma \ref{03.05.2018--1}, we only need to show that $\ph^*(\cq)|_{Y_k}$ and $\ph^*(\cq)|_{Y_K}$ are both trivial. But this is a direct consequence of  Lemma \ref{24.05.2018--1}. 

To verify (ii) we give ourselves 
an object $\cm$ of $(\g T^{\rm tan}_{\ph})_{(k)}$, 
a morphism  $\ps:Z\to X$ in $\g S(X,x_0)$, an object $\cq$ of $(\g T^{\rm tan}_{\ps})_{(k)}$ and an epimorphism $q:\cm\to\cq$. Since $\cm$ and $\cq$ are locally free $\co_{X_k}$-modules \cite[Remarks (a), p.226]{biswas-dos_santos11},  the exact same argument as before proves that $\cq$ lies in $\g T_{\ph,(k)}$. Being a quotient of $\cm$, which belongs to $(\g T^{\rm tan}_{\ph})_{(k)}$, $\cq$ must be in $(\g T^{\rm tan}_{\ph})_{(k)}$. 

The natural arrow ${\rm Gal}'(\ce,\g T,x_0)\to{\rm Gal}'(\ce,\g T_\ph,x_0)$ is  induced by the obvious fully faithful functor $\langle\ce\,;\,\g T_\ph\rangle_\ot\to\langle\ce\,;\,\g T\rangle_\ot$. Now, any   $\cv$ in $\langle\ce\,;\,\g T\rangle_\ot$ is of the form $\cv'/\cv''$, where $\cv''\subset\cv'\subset\bb T^{a,b}\ce$. Since $\nu$ is faithfully flat, Theorem 4.1.2  of \cite{duong-hai18} shows that 
 $\cv'$ and $\cv''$ belong to $\g T_\ph^\circ$; consequently $\cv$ is an object of $\langle\ce\,;\,\g T_\ph\rangle_\ot$ and the fully faithful  functor $\langle\ce\,;\,\g T_\ph\rangle_\ot\to\langle\ce\,;\,\g T\rangle_\ot$ is an equivalence. \end{proof}

In order to make some properties of $\Pi(X,x_0)$ conspicuous, let us make the following definitions. 

\begin{dfn}\label{25.02.2019--1}Let   $G\in\bb{FGSch}/A$ be given. 
 Let $\bb P$ be one of the adjectives ``finite'', ``quasi-finite'' or ``pseudo-finite''. 
We say that $G$ is pro-$\bb P$ (respectively strictly pro-$\bb P$) if there exists a  directed set $I$ and a diagram $\{\nu_{ij}:G_j\to G_i\,:\,i\le j\in I\}$ in $\bb{FGSch}/A$ where each $G_i$ is $\bb P$ (respectively each $G_i$ is $\bb P$ and each $\nu_{ij}$ is faithfully flat) such that $G\simeq \lip_iG_i$. (For the definition of quasi-finite, the reader is directed to \ega{II}{}{6.2.3, p.115}.)
\end{dfn}

\begin{thm}\label{15.03.2019--1}
We suppose that $A$ is Henselian and Japanese,  and moreover assume that $X_K$ is normal.  
\begin{enumerate}
\item[(1)] The group scheme $\Pi(X,x_0)$ is pro-quasi-finite.
\item[(1')]  The group scheme $\Pi(X,x_0)$ is strictly pro-pseudo-finite.   
\item[(2)] 
If in addition $A$ is complete,  then, for each $\ce\in\g T_X^\circ$, the  group scheme $\mm{Gal}'(\ce,\g T_X,x_0)$ is finite and  $\Pi(X,x_0)$ is strictly pro-finite.

\end{enumerate}
\end{thm}
We note that this result shall be improved below---see Theorem \ref{25.03.2019--1}---by a careful application of (2) and Lemma \ref{21.03.2019--1}, but not to make the argument overly involved, we opt for less generality now.

\begin{proof}We consider the set of isomorphism classes $I$ of objects in $\g T^\circ_X$ and order it by decreeing that $\ce\le\cf$ if and only if $\ce\in\langle\cf\rangle_\ot$. Putting $G_\ce:=\mm{Gal}'(\ce,\g T_X,x_0)$, we  obtain a system of group schemes  $\{\nu_{ \ce\cf}:G_\cf\to G_\ce\}$ whose limit is $\Pi(X,x_0)$.

(1) and (1'). We know that for each $\ce\in\g T_X^\circ$ the group   $G_\ce\ot K$ is finite (by Proposition \ref{07.11.2018--1} and Theorem \ref{11.04.2019--2}). Hence, the proof follows from the
\begin{lem}\label{19.03.2019--1} Let $G\in\bb{FGSch}/A$ be such that $G\ot K$ is a pro-finite group scheme. Then $G$ is (a) strictly pro-pseudo-finite, and (b) is pro-quasi-finite.
\end{lem}
\begin{proof}
 Following \cite[Theorem 2.17,p. 989]{duong-hai-dos_santos18}, we write $G=\lip_\al G_\al$ where each $G_\al$ lies in $(\bb{FGSch}/A)$,  the transition morphisms are faithfully flat and  each $G_\al\ot K$ is of finite type over $K$. 
Since $G\ot K$ is pro-finite, $G_\al\ot K$ is in fact finite. But  for a flat $A$-module $M$, the inequality  $\dim_KM\ot K<\infty$ entails $\dim_kM\ot k<\infty$, as we see by lifting a linearly independent set in $M\ot k$ to $M$. We conclude that   $G_\al$ is pseudo-finite. 
In addition, by the same Theorem in \cite{duong-hai-dos_santos18}, we know that $G_\al$ is a projective limit $\lip_iG_{\al,i}$ where $G_{\al,i}$ is flat and of finite type over $A$, and the transition morphisms $G_{\al,j}\to G_{\al,i}$ induce isomorphisms on the generic fibres. In particular, $G_{\al,i}\ot K\simeq G_\al\ot K$, and hence $G_{\al,i}\ot K$  is again finite.  By the same argument as before $G_\al$ is, being of finite type, quasi-finite. 
 
\end{proof}
(2) In view of Corollary \ref{25.06.2018--2} and Proposition \ref{07.11.2018--1}, only the final statement needs proof. 
Due to \cite[Theorem 4.1.2]{duong-hai18} each arrow $\nu_{\ce\cf}:G_\cf\to G_{\ce}$ is  faithfully flat, which is enough argument.
 \end{proof}

We now preset the amplification of Theorem \ref{15.03.2019--1} already mentioned before.

\begin{thm}\label{25.03.2019--1}We suppose that $A$ is Henselian and Japanese,  and moreover assume that $X_K$ is normal.   Then $\Pi(X,x_0)$ is strictly pro-finite. In particular, if $\ce\in\g T_X^\circ$, then $\mm{Gal}'(\ce,\g T,x_0)$ is finite.   \end{thm}
\begin{proof}
Let \[
\Pi(X,x_0)\aro G 
\]
be a faithfully flat morphism to a pseudo-finite flat group scheme over $A$. (The existence of such an arrow is assured by Theorem \ref{15.03.2019--1}.)  
Define \[\te:\rep A{G}\aro \g T_X\] as the composition of the natural functor $\rep A{G}\to
\rep A{\Pi(X,x_0)}$ with a tensor inverse to $\bullet|_{x_0}:\g T_X\to\rep A{\Pi(X,x_0)}$. (That such an inverse exists is proved in \cite[I.4.4.2, p.69]{saavedra72}.)
Since $\Pi(X,x_0)\to G$ is faithfully flat, we conclude $\rep A{G}\to\rep A{\Pi(X,x_0)}$
  is fully faithful (see \cite[Proposition 3.2.1(ii)]{duong-hai18}, for example) and hence that $\te$ is fully faithful.
  
Let $\wh A$ be the completion of $A$ and write $\wh X$ for the base-change $X\ot_A\wh A$; note that 
$\wh X$ is a flat and proper $\wh A$-scheme with geometrically reduced fibres and that 
$x_0$ induces an $\wh A$-point $\wh x_0$ on it. In addition, since $A$ is assumed Japanese, we can say that $\wh K$, the field of fractions of $\wh A$, is a separable extension of $K$ \ega{IV}{2}{7.6.6, p.211}. Consequently, $\wh X\ot_{\wh A} \wh K$ is also a normal scheme \ega{IV}{2}{6.7.4, p.146} and it is then a simple matter to deduce that $\wh X$ is also irreducible so that all properties imposed on   the morphism $X\to \spc A$ in the beginning of this section are valid for $\wh X\to\spc \wh A$.  

Using the base-change functor $\si:\g T_X\to \g T_{\wh X}$ and the equivalence \[\bullet|_{\wh x_0}:\g T_{\wh X}\aro \rep{\wh A}{\Pi(\wh X,\wh x_0)},\] 
we derive a tensor functor  
\[
\tau:\rep AG\aro  \rep{\wh A}{\Pi(\wh X,\wh x_0)}
\]  
preserving forgetful functors (up to tensor natural isomorphism).   
Now, if $i:\rep A{G}\to \rep{\wh A}{G_{\wh A}}$ stands for the base-extension functor, then  
$\tau$ can be prolonged  to a tensor functor 
\[
\xi:\rep {\wh A}{G_{\wh A}}\aro \rep {\wh A}{\Pi(\wh X,\wh x_0)}
\]  
rendering 
\[
\xymatrix{\ar[d]_{i}\rep AG  \ar[rr]^-{\tau}&& \rep{\wh A}{\Pi(\wh X,\wh x_0)}
\\
\rep{\wh A}{G_{\wh A}}\ar[rru]_{\xi}&&
}
\] 
commutative up to natural isomorphism of tensor functors. In addition, $\xi$ preserves the forgetful functors. We contend that the morphism of group schemes induced by $\xi$ is faithfully flat, and for that we rely on Lemma \ref{21.03.2019--1}, whose notations are from now on in force. 

Let   $V\in \g s(G)$.   Since $\te$ is full, we conclude that $H^0(X,\te(V))\simeq A$. 
Hence, flat base-change gives 
$H^0(\wh X,\si\te(V))\simeq \wh A$.
This implies that 
\[ \tau(V)^{\Pi(\wh X,\wh x_0)}
\simeq \wh A.
\]
Likewise, if $M\in\g s_0(G)$, we conclude that 
\[
\tau(M)^{\Pi(\wh X,\wh x_0)}\simeq k.
\]  
Now, as we learn from 1.4 and Proposition 2 of 1.5  in \cite{serre68},  for any $V\in\g s(G_{\wh A})$, 
respectively $M\in \g s_0(G_{\wh A})$, there exists $V^\flat\in \g s(G)$, respectively 
$M^\flat\in \g s_0(G)$,  and an injection $V\to i(V^\flat)$, respectively an injection  $M\to i(M^\flat)$. 
Then, 
\[
\begin{split} \xi(V)^{\Pi(\wh X,\wh x_0)}&\subset\tau(V^\flat)^{\Pi(\wh X,\wh x_0)}\\&\simeq \wh A, 
\end{split}
\]
and 
\[\begin{split} \xi(M)^{\Pi(\wh X,\wh x_0)}&\subset\tau(M^\flat)^{\Pi(\wh X,\wh x_0)}\\&\simeq k. \end{split}
\]
Consequently, Lemma \ref{21.03.2019--1} guarantees that the morphism \[\Pi(X_{\wh A},\wh x_0)\aro G_{\wh A}
\] deduced from $\xi$ is faithfully flat. 
Now, since $\Pi(X_{\wh A},\wh x_0)$ is strictly pro-finite (Theorem \ref{15.03.2019--1}) it is not hard to see that $G_{\wh A}$ is actually finite, so that $G$ must then be finite \cite[I.3.6, Proposition 11, p.52]{bourbaki-ac}.  
\end{proof}

Using \cite[Tag 0AS7]{stacks}, we have:
\begin{cor}\label{16.04.2019--2}Let us adopt the hypothesis of Theorem \ref{25.03.2019--1}. Then, the ring of functions of $\Pi(X,x_0)$ is a Mittag-Leffler $A$-module. \qed
\end{cor}

\section{An application to the theory of torsion points on the Picard scheme}\label{20.03.2019--1}We assume that $A$ is Henselian and Japanese. 
 Let $X$ be an irreducible, proper and flat $A$-scheme with geometrically reduced fibres, and $x_0$ an $A$-point of $X$. Assume in addition that $X_K$ is normal. 
The following result connects Theorem \ref{25.06.2018--4} with the theory of torsion points on abelian schemes.

\begin{prp}\label{11.01.2019--1}Suppose that $A$  has characteristic $(0,p)$ and absolute ramification index  $e$. We give ourselves a positive integer $r$ and an invertible sheaf $\cl$ on $X$. 
\begin{enumerate}\item If $\cl\in\mm{Pic}(X)$ has order $p^r$ and is taken upon reduction to the identity of $\mm{Pic}(X_k)$, then $p^r-p^{r-1}\le e$. 
\item If $Y$ is an abelian scheme over $A$ and  $y\in Y(A)$ is a point of order $p^r$ which reduces to the identity in $Y(k)$, then $p^{r}-p^{r-1}\le e$. 
\end{enumerate}
\end{prp}

The proof requires:

\begin{lem}\label{15.01.2019--2}Let $A$ have characteristic $(0,p)$ and absolute ramification index $e$. Let $r$ be a positive integer. Then, the Neron blowup \cite[Section 1]{waterhouse-weisfeiler80} $\tilde\mu_{p^r}$ of $\mu_{p^r}$ at the origin in the special fibre is finite if and only if 
\[
p^r\le er+\min_{0\le i<r}\{p^i-ie\}.
\]
\end{lem}
\begin{proof}Let us write $\mu_{p^r}=\spc{A[t]/(t^{p^r}-1)}$. Putting $t=1+s$, we have $\mu_{p^r}=\spc A[s]/(\ph)$, where $\ph(s)=s^{p^r}+\sum_{n=1}^{p^r-1}\binom{p^r}{n}s^n$. Now, if $\mm{ord}_p:\ZZ\setminus\{0\}\to\NN$ denotes the $p$-adic valuation, then \[\mm{ord}_p\binom{p^r}{n}=r-\mm{ord}_p(n).
\]
Consequently, writing $\tilde s=\pi^{-1}s$, we obtain  
\[\begin{split}
\ph&=\pi^{p^r}\tilde s^{p^r}+\sum_{n=1}^{p^r-1} \pi^{n+er-e\po\mm{ord}_p(n)}u_n\tilde s^n
\\
&=\pi^{p^r}\tilde s^{p^r} +\sum_{n=1}^{p^r-1}\pi^{\al(n)}u_n\tilde s^n, 
\end{split}
\]
with $u_n\in A^\ti$ and  $\al(n):=er+n-e\po\mm{ord}_p(n)$. 
Now, if $\mm{ord}_p(n)=i$,  then $\al(n)=er+n-ei$ so that $\al(p^i)\le\al(n)$ in this case. Consequently,   
\[
\min_{1\le n<p^r} er+n-e\po\mm{ord}_p(n)=\min_{0\le i<r}er + p^i-ie.
\]
\end{proof}

\begin{proof}[Proof of Proposition \ref{11.01.2019--1}]Let us deal first with the case $r=1$. Using the fact that $\cl$ becomes trivial on the $\mu_p$-torsor associated to it and Theorem \ref{05.11.2018--1}, we see that $\cl\in\g T_X$. 
 Let $G=\mm{Gal}'(\cl)$ and let $\la:G\to\GG_m$ be the associated representation. Then, $\la$ factors through $\mu_{p}\subset\GG_m$ and  $\la\ot K$ induces an isomorphism $G\ot K\to\mu_{p}\ot K$. Since $\la_k$ is trivial, the morphism $\la$ factors through an arrow $\tilde \la:G\to\tilde\mu_{p}$, where  $\tilde\mu_{p}$ is the Neron blowup of $\mu_{p}$ at the identity of the special fibre. Now, because     $G\to \tilde\mu_{p}$ induces an isomorphism between generic fibres and a fortiori an injection among rings of functions, finiteness of $G$ (by    Theorem \ref{25.03.2019--1}) implies that of $\tilde\mu_{p}$ and Lemma \ref{15.01.2019--2} finishes the proof. 

Let us now assume that $r\ge2$ and that the claim is true for all $i<r$. Since $\cl^{\ot p^{r-i}}$ has order $p^{i}$, we conclude that 
\begin{equation}\label{15.01.2019--1}
p^{i}\le e+p^{i-1}
\end{equation} 
for $i\in\{1,\ldots,r-1\}$. Using  the fact that $\cl$ becomes trivial on the $\mu_{p^r}$-torsor associated to it and Theorem \ref{05.11.2018--1}, we see that $\cl\in\g T_X$.  Let $G=\mm{Gal}'(\cl)$ and let $\la:G\to\GG_m$ be the associated representation. 
Just as for the particular case, we conclude that the blowup of $\mu_{p^r}$ at the identity of the special fibre is finite. Hence, by Lemma \ref{15.01.2019--2},
\[
p^r\le er+\min_{0\le i<r}\{p^i-ei\}.
\] 
But using \eqref{15.01.2019--1}, we see that 
\[1\ge p-e\ge p^2-2e\ge\cdots\ge p^{r-1}-(r-1)\po e.\]
Consequently, 
$\displaystyle\min_{0\le i<r}\{p^i-ei\}=p^{r-1}-(r-1)\po e$, so that 
\[p^r\le er+p^{r-1}-(r-1)\po e,\]
and this is what we wanted.

(2) One takes $X$ to be an abelian scheme such that $Y$ is its dual and applies part (1).
\end{proof}

\begin{rmk}Corollary \ref{11.01.2019--1}(2) can be deduced from ``Cassels'  Theorem'' on formal groups. The one dimensional case is classical \cite[IV.6.1, p.123]{silverman86}, while the higher dimensional can be found in \cite[p.966]{grant13}. 
\end{rmk}

\section{Application to the theory of  torsors}\label{15.03.2019--2} 
We suppose that $A$ is Henselian and Japanese. Let $X$ be an irreducible, proper and flat $A$-scheme with geometrically reduced fibres, and $x_0$ an $A$-point of $X$. 

\begin{thm}\label{14.01.2019--1} Let $G\in(\bb{FGSch}/A)$ be  finite  and $\ph:Q\to X$ be a   $G$-torsor. Then, there exists a $\ps:Z\to X$ in $\g S^+(X,x_0)$ such that $\te_Q:\rep AG\to\bb{coh}(X)$ takes values in $\g T_\ps$. 
\end{thm}

\begin{proof}According to  Theorem \ref{05.11.2018--1}, there exists $\ps:Z\to X$ in $\g S^+(X,x_0)$ and a commutative diagram 
\[
\xymatrix{Z\ar[r]\ar[dr]_\ps&Q\ar[d]^\ph\\&X.}
\]
Hence, for each $M\in\rep AG$ we conclude that $\ps^*(\te_Q(M))$ is trivial relatively to $A$ since $\ph^*\te_Q(M)\simeq\co_Q\ot_AM$ is trivial relatively to $A$. 

\end{proof}

In \cite[87ff]{nori82}, Nori defines the notion of ``reduced torsor'' in order to understand which group schemes do come as quotients of his fundamental group. We follow the same idea here, but instead of starting off with Nori's definition, we prefer to use an equivalent characterization \cite[Proposition 3, p.87]{nori82}. 

Before reading the definition to come, the reader might profit to recall that, for any $G\in(\bb{FGSch}/A)$ and any $G$-torsor $\ph:Q\to X$ above $X$, the functor $\te_Q:\rep AG\to \bb{coh}(X)$ is exact and faithful since $\ph^*\te_Q$ is naturally isomorphic to $M\mapsto \co_Q\ot_AM$ (see for example the proof of (a) in \cite[Part 1, Proposition 5.9]{jantzen87}).

\begin{dfn}\label{05.07.2018--1}Let $G\in(\bb{FGSch}/A)$ be  finite  and $Q\to X$ be a   $G$-torsor having an $A$-point $q_0$ above $x_0$. We say that the data $(Q,G,q_0)$ defines a Nori-reduced torsor if $\te_Q:{\rm Rep}_A(G)\to\bb{coh}(X)$ is a fully faithful functor.  
\end{dfn}

\begin{prp}[{compare to \cite[Theorem 7.1]{mehta-subramanian13}}]\label{29.11.2018--2}Let $G\in(\bb{FGSch}/A)$ be  finite  and $Q\to X$ be a   $G$-torsor having an $A$-point $q_0$ above $x_0$.  Then, the following conditions are equivalent: 
\begin{enumerate}[(i)]
\item[(i)] The triple $(Q,G,q_0)$ is Nori-reduced. 
\item[(ii)] The ring of global functions of $Q_k$ is $k$. 
\item[($ii'$)] The triple $(Q_k,G_k,q_{0,k})$ is reduced in the sense of \cite[Definition 3, p.87]{nori82}. 

\item[(iii)] The $A$-scheme $Q$ is $H^0$-flat and $A=H^0(Q,\co_Q)$. 
\end{enumerate}
\end{prp}

\begin{proof}The proof relies on the fact that $A[G]_{\rm left}$ is an algebra in the category $\rep AG$, and that the corresponding $\co_X$-algebra $\te_Q(A[G]_{\rm left})$ is simply $\co_Q$. 

$(i) \Rightarrow (ii)$. 
Because the functor $\te_Q$ is fully faithful and $\te_Q(k[G]_{\rm left})\simeq\co_{Q}\ot k$, we conclude that $k=H^0(X,\co_{Q}\ot k)$. But $H^0(X,\co_{Q}\ot k)=H^0(Q,\co_{Q}\ot k)$, and hence the only regular functions on the $k$-scheme $Q_k$ are the constants. 

$(ii)\Leftrightarrow(ii')$. This is  \cite[II, Proposition 3]{nori82}.

$(ii)\Rightarrow(iii)$. Recall that $Q$ is flat over $X$ and a fortiori over $A$. Now, because $k=H^0(Q,\co_Q\ot_A k)$, we can employ Proposition  12.10 of Chapter III in \cite{hartshorne77} to conclude that  $Q$ is  cohomologically flat of degree zero over $A$. 
Since $H^0(\co_Q)$ is a finite and flat $A$-module, the isomorphism $k\simeq k\ot_A H^0(\co_Q)$ proves that $A\simeq H^0(\co_Q)$. 

$(iii)\Rightarrow (i)$. 
Let $\ph$ denote the structural morphism $Q\to X$; by assumption it belongs to $\g S^+(X,x_0)$.   If we agree to write $H:=\Pi(X,\ph,x_0)$ (see Definition \ref{14.01.2019--2}), the existence of the point $q_0$ gives an isomorphism between $(\bullet|_{x_0})\circ\te_Q$ and the forgetful functor $\rep AG\to\modules A$
and hence  a morphism of   group schemes 
\[
\rho:H\aro G
\] together with a commutative diagram 
\begin{equation}\label{28.11.2018--2}
\xymatrix{
\rep AG\ar[rr]^{\te_Q} \ar[d]_{\rho^\#}&&\ar[dll]_{\sim}^{\bullet|_{x_0}}  \g T_\ph\\
\rep A H .&&}
\end{equation}
Now, \[\begin{split}k&= H^0(X,\co_{Q_k})\\
&\simeq H^0(X,\te_Q(k[G]_{\rm left}))\\&\simeq   \left(   k[G]_{\rm left}         \right)^{H}\\&\simeq (k[G]_{\rm left})^{H_k}.
\end{split}
\]
According to Lemma \ref{29.11.2018--1}, this is only possible when $\rho_k:H_k\to G_k$ is faithfully flat. Analogously,  we have that $(A[G]_{\rm left})^{H}\simeq A$. This implies that $(K[G]_{\rm left})^{H_K}\simeq K$, and hence  $\rho_K$ is faithfully flat according to Lemma \ref{29.11.2018--1}.  In conclusion, $\rho_k$ and $\rho_K$ are faithfully flat, and hence   $\rho$ must be faithfully flat  \cite[4.1.1, p. 1124]{duong-hai18}. Together with \cite[3.2.1(ii), p. 1121]{duong-hai18}, we conclude that $\rho^\#$ is fully faithful, so that diagram \eqref{28.11.2018--2} secures fully faithfulness of  $\te_Q$.
\end{proof}
 
We shall now keep the notation and assumptions of Proposition  \ref{29.11.2018--2} and  offer other properties equivalent to the ones in its statement. This will allow us, in passing, to  render the connection with \cite[Theorem 7.1]{mehta-subramanian13} and to \cite[Definition 3, p.87]{nori82} more transparent. First we develop some preliminaries.

\label{29.03.2019--5}Let us abbreviate $\Pi=\Pi(X,x_0)$. Similarly to \cite{nori} (see \S2 and the argument on p.39), there exists a $\Pi$-torsor 
\[
\wt X\aro X
\] 
with an $A$-point $\tilde x_0$ above $x_0$ such that $\te_{\wt X}\circ(\bullet|_{x_0})\simeq{\rm id}$ and $(\bullet|_{x_0})\circ\te_{\wt X}\simeq{\rm id}$ as tensor functors. The quasi-coherent $\co_X$-algebra of $\wt X$ is a direct limits of coherent modules belonging to $\g T_X$ and corresponds, in $\mm{Ind}\,\rep A{\Pi}$, to $A[\Pi]_{\rm left}$, see \cite[Definition, p. 32]{nori}. The  torsor $\wt X$ is called the {\it universal pointed torsor}.

Let  $\rho:G'\to G$ be an arrow of $(\bb{FGSch}/A)$.  We say that $Q$ has a {\it reduction of structure group to $\rho$, or to $G'$},  if there exists a $G'$-torsor $Q'\to X$ together with an isomorphism  $Q'\ti^{\rho}G\to Q$. In addition, if $Q'$ can be picked to come with an $A$-point $q_0'$ such that $(q_0',e)$ corresponds to $q_0$ under the aforementioned isomorphism,  then the reduction is called {\it pointed}.  Note that,  we  do not assume $\rho$ to be a closed embedding. 

\begin{cor}\label{27.03.2019--1} The equivalent properties appearing in Proposition \ref{29.11.2018--2} are also equivalent to each one of the following conditions: 
\begin{enumerate}[(a)]\item  If $Q$ has a pointed reduction to $\rho:G'\to G$, then $\rho$ is faithfully flat. 
\item There exists a faithfully flat morphism $\rho:\Pi\to G$ and 
an isomorphism of pointed $G$-torsors $\wt X\ti^\rho G\stackrel\sim\to Q$. (That is, $\wt X$ defines a pointed reduction of $Q$.)
\end{enumerate}
\end{cor}  
\begin{proof} Proposition \ref{29.11.2018--2}-$(i)$ $\Rightarrow$ $(a)$.  Let  $\rho:G'\to G$ define a pointed reduction $Q'\to X$; it then follows that $\te_{Q'}\circ\rho^\#$ is isomorphic to $\te_Q$. Since $\te_{Q'}$ is faithful, we conclude that $\rho^\#:\rep AG\to\rep A{G'}$ is full and faithful. By  Lemma \ref{21.03.2019--1},    $\rho$ is faithfully flat. 

$(a)$ $\Rightarrow$ $(b)$.  
According to Theorem \ref{14.01.2019--1},  $\te_Q$ takes values in $\g T_X$; the existence of the point $q_0$ allows us to say that $\bullet|_{x_0}\circ\te_Q$ is isomorphic to the forgetful functor $\rep AG\to\modules A$ which gives us an arrow $\rho:\Pi\to G$ such that $\rho^\#\simeq \bullet|_{x_0}\circ\te_Q$ (as tensor functors). Hence, $\te_{\wt X}\circ\rho^\#\simeq \te_Q$, and we conclude, as in \cite[Proposition 2.9(c)]{nori}, that $Q$ has a pointed reduction to $\rho$. But $(a)$  forces $\rho$ to be faithfully flat. 

$(b)$ $\Rightarrow$ Proposition \ref{29.11.2018--2}-$(i)$. We know that $\te_Q\simeq \te_{\wt X}\circ\rho^\#$ in this case; but, by construction, $\te_{\wt X}$ is fully faithful  as is $\rho^\#$ (by \cite[3.2.1(ii), p.1121]{duong-hai18}, say). Therefore, $\te_Q$ is fully faithful. 
\end{proof}

\begin{rmk}
Let $Q\to X$ be as in the statement of Proposition \ref{29.11.2018--2}. The condition that $\te_Q:\repo AG\to\g T_X^\circ$ be full is  not enough to assure that $G$ is a faithfully flat quotient of $\Pi$ (so that this is missing in \cite[Theorem 7.1]{mehta-subramanian13}). 
 \end{rmk}

\section{Essentially finite vector bundles on the fibres: Reviewing  a theory of of Mehta and Subramanian }\label{06.03.2019--1} Let $X$ be an irreducible, projective and flat $A$-scheme with geometrically reduced fibres, and $x_0$ an $A$-point of $X$.   
(Recall that over a perfect field, an algebraic scheme is geometrically reduced if and only if it is reduced \ega{IV}{2}{4.6.1, p.68}. Note also that $X$ must be reduced.) 

The following result, which is one of the main points in \cite{mehta-subramanian13}, is essentially a consequence of the method employed by Deninger and Werner in proving  \cite[Theorem 17, p.573]{deninger-werner05} plus Section \ref{FGS}. Before putting forth its statement, let us recall the notion of an $F$-trivial vector bundle. 

If $M$ is a proper scheme over an unspecified \emph{perfect} field of positive characteristic, a vector bundle $E$ on $M$ is called  $F$-trivial \cite[Section 2,p. 144]{mehta-subramanian02}  if for a certain   $s\in\NN$, the pull-back of $E$ by a geometric Frobenius morphism 
${\rm Fr}^s:M^{(-s)}\to M$ is trivial.

\begin{thm}\label{14.05.2018--1}Suppose that $A$ is Henselian and Japanese, and that $k$ is perfect of characteristic $p>0$. 
Let $E$ be an $F$-trivial vector bundle on $X_k$. Then, there exists a proper and surjective morphism $\ps:Z\to X$
such that:  
\begin{enumerate}[(1)]\item The ring $B:=H^0(\co_Z)$ is a discrete valuation ring and is a finite extension of $A$. 
\item The canonical morphism $Z\to\spc B$ is flat and has geometrically reduced fibres. 
\item The $B$-scheme $Z$ has a $B$-point above $x_0$. 
\item Write $\ell$ for the residue field of $B$ and denote by \[\ps_0:Z\otu B\ell\aro X_k\] 
the morphism of $k$-schemes naturally induced by $\ps$.
Then  $\ps_0^*(E)$ is trivial.
\end{enumerate} 
\end{thm}

\begin{proof}
{\it The case of characteristic $(0,p)$.}  
We assume that $E$ is trivialized by ${\rm Fr}^s:X_k^{(-s)}\to X_k$. Let $X\to\PP_A^n$ be a closed immersion. 
Write $\Ph:\PP_A^n\to\PP^n_A$ for the evident $A$-morphism   lifting  the {\it $k$-linear} Frobenius morphisms ${\rm Fr}^s:\PP_k^n\to\PP_k^n$, and consider the cartesian diagram 
\[
\xymatrix{Y\ar[r]^\ph\ar@{^{(}->}[d]\ar@{}[dr]|\square&X\ar@{^{(}->}[d]\\\PP_A^n\ar[r]_{\Ph}&\PP_A^n.
}
\]
We note that $\Ph$ is a finite, flat and surjective morphism, so that $\ph:Y\to X$ is likewise; in particular this implies that $Y$ is $A$-flat. 

Base-changing by means of $A\to k$ we get the cartesian diagram 
\[
\xymatrix{
Y_k\ar@{}[dr]|\square\ar@{^{(}->}[d]\ar[r]^{\ph_k}&X_k\ar@{^{(}->}[d]
\\
\PP_k^n\ar[r]_{{\rm Fr}^s}&\PP^n_k 
}
\]
so that, since $X_k^{(-s)}$ is reduced, there exists a closed embedding \[
j:X_k^{(-s)}\aro Y_k
\] 
which {\it identifies $X_k^{(-s)}$ with $Y_{k,\,{\rm red}}$} and, in addition, produces a   factorisation of ${\rm Fr}^s:X_k^{(-s)}\to X_k$ like so 
\[
\xymatrix{X_k^{(-s)}\ar@{^{(}->}[dr]_j\ar@/^2pc/[drr]^{{\rm Fr}^s}&&\\
&Y_k\ar@{}[dr]|\square\ar@{^{(}->}[d]\ar[r]^{\ph_k}&X_k\ar@{^{(}->}[d]
\\
&\PP_k^n\ar[r]_{{\rm Fr}^s}&\PP^n_k .
}
\]
See Lemma 19 in \cite{deninger-werner05}.
Consequently, if $V$ is any \emph{reduced} scheme and $\al:V\to Y_k$ is a $\ZZ$-morphism, we   conclude that  $\al^*\ph_k^*E$ is trivial since $\al$ factors as 
\[V\aro X_k^{(-s)}\arou j  Y_k.\] 

A direct application of Theorem \ref{05.11.2018--1} now gives us a commutative diagram 
\[
\xymatrix{Z\ar[r]\ar[dr]_\ps&Y\ar[d]^\ph\\&X
}
\]
such that 
\begin{enumerate}[(i)]\item the morphism $\ps$ is surjective and proper. 
\item The ring $B:=H^0(\co_Z)$ is a discrete valuation rings and is a finite extension of $A$. 
\item The canonical morphism $Z\to\spc B$ is flat and has geometrically reduced fibres. 
\item The $B$-scheme $Z$ has a $B$-point above $x_0$. 
\end{enumerate}

In this situation, the proof is concluded by the observation preceding it. Indeed, if $\ell$ is the residue field of $B$ and
$\ps_0:Z\ot_B\ell\to X_k$
is the  arrow induced by $\ps$, we conclude that  $\ps_0^*(E)$ is trivial because $Z\ot_B\ell$ is reduced so that $\ps_0$ factors through 
$Z\ot_B\ell\to Y_k$.

\vspace{.2cm}\noindent\emph{Proof in the case of characteristic $(p,p)$.} The idea behind the proof is much simpler, but notation and technicalities hinder its handling.  

Suppose that $E\in\bb {VB}(X_k)$ becomes trivial after being pulled back by  $F_{X_k}^s:X_k\to X_k$.
Employing the commutative diagram of $\FF_p$-schemes 
\[
\xymatrix{X\ar[r]^{   F_X^s}\ar[d]&X\ar[d]
\\
\spc A\ar[r]_{F_A^s}&\spc A,
}
\]
we see that the morphism $(F^s_X)_0:X_k\to X_k$ induced on special fibres is none other than $F_{X_k}^s$. 
Hence, if  $F_A^s$ is a finite morphism, the choice  $Z=X$ and $\Ps=F_X^s$ is sufficient to fulfill all but condition (3) of the statement. But finiteness of $F_A$ is not always assured, and we choose to argue as in \ega{IV}{3}{\S8}. 

Let \[\La=\left\{\begin{array}{c} \text{$B$ is a d.v.r. dominated by $A$ and dominating $A^{p^s}$,  } \\\text{and such that ${\rm Frac}\,B$ is a finite extension of $K^{p^s}$} 
\end{array}\right\},
\] 
and endow it with the   partial order defined by domination of d.v.r.'s. 
As $A$ is Japanese,  for any $B\in\La$, the $A^{p^s}$-module $B$ is finite, and 
any element in $A$ belongs to some $B\in\La$ (see Theorem 10.2 and Exercise 11.2 in  \cite{matsumura}). Consequently, the limit $\lip_{B\in\La}\spc B$ in the category of $A$-schemes is simply $F_A^s:\spc A\to\spc A$.  
Employing \ega{IV}{3}{8.8.2-ii, p.28}, there exists $\g o\in\La$ and a $\g o$-scheme of {\it finite type} $Y$ 
fitting into a cartesian commutative diagram 
\begin{equation}\label{30.10.2018--2}
\xymatrix{
X\ar[r]^u\ar[d]\ar@{}[dr]|\square & Y\ar[d]
\\
\spc A\ar[r]&\spc \g o.
}
\end{equation}
In addition, if $u_{B}:X\to Y\ot_{\g o}B$ stands for the canonical morphism, an application of \ega{IV}{3}{8.2.5,p.9} shows that 
\[
(u_B):X\aro \lip_{B\ge \g o}Y\otu{\g o}B 
\]
is in fact an isomorphism of $\g o$-schemes. Also by loc.cit., the canonical morphism 
\[
X\otu{A,F_A^s}A \aro  \lip_{B\ge\g o } X\otu{A,F_A^s}B 
\]
is also an isomorphism. 
The relative Frobenius morphism 
\[
\g f: X\aro  X\otu{  A,F_A^s} A
\]
now gives rise, via \ega{IV}{3}{8.8.2-i,p.28}, to a $B\ge\g o$ and a morphism of $B$-schemes 
\[
f:Y\otu{\g o} B\aro X\otu {A,F_A^s}B, 
\]
such that $f\ot_{B}A$ corresponds to $\g f$. Hence, if $\ps$ stands for the composition of $f$ with the projection $X\ot_{A,F_A^s}B\to X$, we arrive at a commutative diagram 
\begin{equation}\label{30.10.2018--3}
\xymatrix{
Y\otu{\g o}B\ar[r]^-{\psi}\ar[d] & X\ar[d]
\\
\spc B\ar[r]&\spc A.
}
\end{equation}
In addition, 
\[
\psi\circ u_B={  F}_X^s.
\] 
Let us agree to write $Z=Y\ot_{\g o}B$. Then,   paralleling diagram \eqref{30.10.2018--2}, we have 
\begin{equation}\label{30.10.2018--4}
\xymatrix{
X\ar[r]^-{u_B}\ar[d]\ar@{}[dr]|\square & Z\ar[d]
\\
\spc A\ar[r]&\spc B.}
\end{equation}

\noindent\emph{Claim.} The following statements are true. 
\begin{enumerate}[(i)]
\item The morphism $\psi$ is finite and surjective.  
\item As a $B$-scheme, $Z$ is flat and proper. 
\item The geometric fibres of $Z$ over $B$ are reduced. 
\item The ring of global functions of $Z$ is $B$. 
\item Write  $\psi_0$ for the morphism induced from $\psi$ between special fibres. Then $\psi_0^*(E)$ is trivial. 
\end{enumerate}

\noindent{\it Proof.} 
(i) Surjectivity follows from $F_X^s=\ps\circ u_B$. 
Because $\g f$ is finite and the inclusion $B\to A$ is faithfully flat, we conclude that $f$ is finite \ega{IV}{2}{2.7.1, p.29}. Hence, $\ps$ is finite as $F_A^s:A\to B$ is finite.

(ii) We note that the morphism $\spc A\to\spc B$ in diagram \eqref{30.10.2018--4} is faithfully flat. Consequently, the claim is proved by employing \ega{IV}{3}{2.5.1, p. 22},  \ega{IV}{3}{2.7.1, p.29} and the fact that $X$ is flat and proper over $A$.

(iii) This is a direct consequence of diagram \eqref{30.10.2018--4} and the fact that being geometrically reduced is independent of the field extension \ega{IV}{2}{4.6.10, p.70}. 

(iv) This is a direct consequence of   flat base-change applied to diagram \eqref{30.10.2018--4} and $A=H^0(\co_X)$ (which follows from Lemma \ref{03.07.2018--2}, say). 

(v) Let now $\ell$ be the residue field of $B$; it is clear that $B\to A$ in fact induces an isomorphism $\ell\stackrel\sim\to k$. Since $X=Z\ot_BA$, it follows that  $(u_B)_0:X\ot_Ak\to Z\ot_B\ell$ is also an isomorphism. Hence, $\ps_0^*(E)$ is trivial because $\ps u_B=F_X^s$ so that  $(u_B)_0^*(\ps_0^*(E))$ is trivial. The claim is proved.

To finish the proof, we note that $\ps:Z\to X$ satisfies all the conditions in the statement of the Theorem except for the existence of a $B$-point above $x_0$. Now, the inverse image $\ps^{-1}(x_0)$ comes with a finite and surjective morphism to $\spc A$ (for surjectivity, see \ega{I}{}{3.5.2,p.115}). Hence, it is possible to find a {\it finite} extension of d.v.r.'s $B'\supset A$ and a point $\spc B'\to\Ps^{-1}(x_0)$ which then gives a commutative diagram  
\[
\xymatrix{
&Z\ar[d]
\\
\spc B'\ar[r]\ar[ru]^{z_0'} & \spc B}
\]
such that $z_0'$ is a $B'$-point of $X$ above   $x_0$ and the induced arrow $B\to B'$ is a finite extension. Consequently, letting $Z'$ be $Z\ot_BB'$ and $\ps':Z'\to X$  the composition $Z'\stackrel{\rm pr}\to Z\stackrel\ps\to X$, we see that $\ps':Z'\to X$ now satisfies all properties in the statement of the theorem. 
\end{proof}

As a consequence of Theorem \ref{14.05.2018--1}, we can give a simple alternative description of the vector bundles in  $\g T$ and, in doing so,  connect our theory to that of \cite{mehta-subramanian13}. See   Corollary \ref{23.11.2018--2}. 

\begin{thm}[Compare to {\cite[Lemma 3.1]{mehta-subramanian13}}]\label{MS_theorem}Suppose that $A$ is Henselian and Japanese,   that $k$ is perfect, and that $X$ is in addition normal.  
Let $\ce\in{\bf VB}(X)$ be such that $\ce_K$ and $\ce_k$ are essentially finite. Then, there exists a proper and surjective morphism \[\ze:X'\aro X\] such that 
 
\begin{enumerate}[(1)]
\item The ring $A':=H^0(\co_{X'})$ is a discrete valuation ring and a finite extension of $A$. 
\item The canonical morphism $X'\to\spc A'$ is flat and has geometrically reduced fibres. 
\item The $A'$-scheme $X'$ has an $A'$-point above $x_0$. 
\item The vector bundle $\ze^*(\ce)$ is trivial. 
\end{enumerate} 
\end{thm}

\begin{proof}As the vector bundle $\ce_k$ is essentially finite, it is possible to find a torsor under an etale group scheme  $f:Y\to X_k$, and a fortiori an etale covering, such that $f^*(\ce_k)$ is $F$-trivial see \cite[\S3]{nori}. In addition, $Y$ can be chosen to posses a $k$-rational point $y_0$ above $x_{0,k}$ (cf. loc.cit) and to satisfy $k=H^0(Y,\co_Y)$ \cite[II, Proposition 3, p. 87]{nori82}.  
Now, as $A$ is Henselian, Theorem 3.1 on p.30 of \cite{artin69} (the remarkable equivalence, Grothendieck's existence theorem and Artin approximation) allows us to find an  etale covering $\tilde f:\wt Y\to X$ lifting $Y\to X_k$. Looking at the finite and etale $A$-scheme $\tilde f^{-1}(x_0)$, applying one of the main properties  of a Henselian local ring \cite[VII.3, Proposition 3, p.76]{raynaud70}, and making use of the $k$-point $y_0:\spc k\to \tilde f^{-1}(x_0)$, we can find an $A$-point $\tilde y_0$ in $\wt Y$ above $x_0$. 
Note that $\wt Y$ inherits the following properties from $X$:
it is flat, proper, and has geometrically reduced fibres over $A$, and it is normal.  In addition, $\wt Y$ is connected and its normality then  assures  irreducibility. (Irreducibility might fail without   normality.)
 Therefore, $\wt Y$ satisfies all hypothesis imposed on $X$ in the beginning of the section and in the statement of the theorem. Note that, by construction,  the restriction of $\wt f^*\ce$ to ${\wt Y_k}$ is $F$-trivial. 

Because of the previous paragraph, we  suppose, so to lighten notation, that $\ce_k$ is $F$-trivial already on $X_k$. Let us apply Theorem \ref{14.05.2018--1} to the vector bundle $E=\ce_k$. Then, keeping with the notations of this theorem, 
we conclude that 
\[\ps^*(\ce)|_{Z\otu B\ell}\]
is trivial.  
Let $L={\rm Frac}(B)$.
Since $\ps^*(\ce)|_{Z\ot_BL}$ is an essentially finite vector bundle,  \cite[\S3]{nori} assures that we  can find a  torsor with  finite structural group 
\[\la^\circ:
Q^\circ\aro Z\otu BL
\] 
such that 
\[
\la^{\circ*}\left(\ps^*(\ce)|_{Z\ot_BL}\right)
\]
is trivial. In addition, $Q^\circ$ might be chosen to come with two extra properties, which are: 
\begin{itemize}\item Letting $z_0:\spc B\to Z$ be the point above $x_0$ mentioned in Theorem \ref{14.05.2018--1}, 
$Q^\circ$ has an $L$-rational point $q_0^\circ$  above $z_{0,L}$. (This is not used in what follows.)
\item The ring of global functions of $Q^\circ$ is $L$, see Proposition 3 of Chapter II, p.87, in \cite{nori82}. In particular $Q^\circ$ is connected. 
\end{itemize}

Let $Q^\Box\to Q^\circ$ be the associated reduced scheme and write 
\[
\la^\Box:Q^\Box\aro Z\otu B L
\] 
for the induced morphism. Clearly $Q^\Box$ is connected, $\la^\Box$ is surjective, finite and
\[
\la^{\Box*}\left(\ps^*(\ce)|_{Z\ot_BL}\right)
\] 
is trivial.

Let $\mu:Z\ot_BL\to Z$ be the natural immersion and write 
\[
\la:Q\aro Z
\]
for the integral closure of the quasi-coherent $\co_Z$-algebra $(\mu\la^\Box)_*(\co_{Q^\Box})$, see \ega{II}{}{6.3, 116ff} or \cite[Tag 035H]{stacks}. 
This means, according to \ega{II}{}{6.3.4, p.117}, that for each affine open subset $V\subset Z$, 
the ring $\co_Q(\la^{-1}V)$ is the integral closure of $\co_{Z}(V)$ inside $\co_{  Q^\Box}((\mu \la^{\Box})^{-1}(V))$.
In particular, $Q$ is   flat as a $B$-scheme and 
\[
Q\otu {B} L= Q^\Box.  
\]
Since $A$ is universally Japanese \ega{IV}{2}{7.7.2, p. 212}, for each affine and open subset $V$ of $Z$, the ring $\co_Z(V)$ is universally Japanese and noetherian, and hence is a Nagata ring \cite[Tag 032R]{stacks}. As a consequence of \cite[Tag 03GH]{stacks} and the fact that $ Q^\Box$ is reduced, 
 we see that $\la$ is a \emph{finite} morphism.  Because $\la(Q)$ contains $Z\ot_BL$ (recall that $\la^\Box$ is surjective), we conclude that $\la$ is surjective. 
Finally, both $(\ps\la)^*(\ce)|_{Q\ot_B\ell}$ and  $(\ps\la)^*(\ce)|_{Q\ot_BL}$ are trivial. 

Theorem \ref{05.11.2018--1} can be applied to $\ps\la:Q\to X$, and this allows us to find a commutative diagram of schemes 
\[
\xymatrix{R\ar[r]  \ar[dr]_\te& Q\ar[d]^{ \ps\la}  
\\
&X, } 
\]
such that: 
\begin{itemize}\item The morphism $\te$ is proper and surjective. 
\item The ring $C:=H^0(\co_R)$ is a discrete valuation ring   and a finite extension of $A$. 
\item The canonical morphism $R\to \spc C$ is flat and has geometrically  reduced fibres.
\item The $C$-scheme $R$ has a $C$-point above $x_0$.  
\end{itemize}
As the natural arrow $\spc C\to\spc B$ sends the generic, respectively special, point to the generic, respectively special,  point, triviality of  $(\ps\la)^*(\ce)|_{Q\ot_B\ell}$ and  $( \ps\la)^*(\ce)|_{Q\ot_BL}$ allows us to conclude that  the restrictions of $\te^*\ce$ to the generic and special fibres of $R$ over $C$ are trivial.  Because   $R$ is $H^0$-flat over $C$ (it has reduced fibres), we conclude by employing Lemma \ref{03.05.2018--1}  that $\te^*\ce$ is trivial. 
\end{proof}

For the sake of discussion, let us make the following: 
\begin{dfn}\label{23.11.2018--3} The Mehta-Subramanian category of $X$, denote it $\bb{MS}(X)$, is the full subcategory of $\bb{VB}(X)$ whose objects are 
\[
\{\ce\in\bb{VB}(X)\,:\,\text{$\ce_k$ and $\ce_K$ are essentially finite}\}.
\]
\end{dfn}

An immediate consequence of \cite[Theorem 1]{antei-mehta11} (or \cite[Theorem I]{tonini-zhang17}) and Theorem \ref{MS_theorem} is then: 

\begin{cor}\label{23.11.2018--2} Suppose that $A$ is Henselian and Japanese,   that $k$ is perfect, and that in addition to the hypothesis in the beginning of the section, $X_k$ and $X_K$ are  normal. Then $\g T^\circ=\bb{MS}(X)$. 
\end{cor}
\begin{proof}Let $\ce\in\bb{MS}(X)$. Since $X$ is flat over $A$, normality of $X_k$ {\it and} $X_K$ implies normality of $X$ \ega{IV}{2}{6.5.4, p.143} and  Theorem \ref{MS_theorem} may be applied. Consequently, $\ce$ belongs to $\g T^\circ$ as $\ze$ in Theorem \ref{MS_theorem} lies in $\g S^+(X,x_0)$.  Conversely, let $\ce$ be a vector bundle in $\g T$. Since $X_k$ and $X_K$ are normal, we know  that $\ce_k$ and $\ce_K$ are essentially finite (due to \cite[Theorem 1]{antei-mehta11} or \cite[Theorem I]{tonini-zhang17}).  
\end{proof}

\begin{rmk}\label{28.11.2018--1}Since a point $x\in X$ above the generic fibre specializes to a point on the special fibre, normality of  $X$ is equivalent to normality of $X$ {\it on the points of $X_k$}. Of course,  $X_k$ can easily fail to be normal even when $X$ is regular.
\end{rmk}

\section{Further applications to the theory of torsors}\label{14.01.2019--3}
We assume that $A$ is Henselian, Japanese and has a perfect residue field. 
Let $X$ be an irreducible, projective and flat $A$-scheme with geometrically reduced fibres, and $x_0$ an $A$-point of $X$.


\begin{thm}[{compare to \cite[Corollary 3.2]{mehta-subramanian13}}]\label{25.06.2018--1} 
Let us add to the assumptions made at the start of this section that $X$ is normal.
Let $G\in({\bf FGSch}/A)$ be quasi-finite over $A$,  
\[
 Q\aro X
\]
be a $G$-torsor, and  $q_0$ an $A$-point of $Q$ above $x_0$. 

\begin{enumerate}[(1)]\item There exists $\ze:X'\to X$ in $\g S^+(X,x_0)$ (see Definition \ref{26.06.2018--1}) and $\ce\in\g T_\ze^\circ$ such that $\te_Q:\rep AG\to\bb{coh}(X)$ takes values in $\langle\ce\,;\,\g T_\ze\rangle_\ot$ (and a fortiori in $\g T_X$). 
\item 
There  exists a finite $H\in(\bb{FGSch}/A)$, a morphism $\rho:H\to G$, 
an $H$-torsor $R\to X$ and an $A$-point $r_0:\spc A\to R$ together with an isomorphism of torsors  
\[
R\ti^{H}G\arou\sim Q
\]
  sending the $A$-point  $(r_0,e)$ of $R\ti^{H}G$ to  $q_0$. In addition, it is possible to choose $\rho$ to be a closed immersion. 
Put differently, $Q$ has a reduction of structure group to a finite group scheme.

\item If $H^0(Q,\co_Q)\simeq A$, then $G$ is in fact finite.  
\end{enumerate}

\end{thm}

\begin{proof}As is well-known (by adapting the proofs in \cite[3.3]{waterhouse79}) the facts that   $G$ is of finite type and $A$ is a d.v.r. allow us to find $E\in\repo AG$ such that the resulting morphism $G\to\bb{GL}(E)$ is a closed immersion, or, in the terminology of \cite[\S3]{duong-hai-dos_santos18}, $E$ is a faithful representation. We then write $\ce=\te_Q(E)$ and note that since $Q_k\to X_k$ and $Q_K\to X_K$ are {\it finite} principal bundles, the vector bundles $\ce_k=\te_{Q_k}(E_k)$ and $\ce_K=\te_{Q_K}(E_K)$ are in fact essentially finite \cite[Proposition 3.8, p.38]{nori}.  
 Let $\ze:X'\to X$, $A'$, and $x_0'$ be as in Theorem \ref{MS_theorem} when applied to $\ce$. 
Note that $\ze:X'\to X$ is     in $\g S^+(X,x_0)$ and that $\ce\in\g T_\ze$.

\bigskip\noindent (1).
We write $Q'$ for the $G$-torsor $X'\ti_XQ$. We know that for each $M\in\rep AG$, the coherent $\co_{X'}$-module  $\ze^*\te_Q(M)$ is isomorphic to $\te_{Q'}(M)$. Because each $\te_Q(\bb T^{a,b}E)$ belongs to $\langle\ce\,;\,\g T_\ze\rangle_\ot$, we conclude that     each  $\te_{Q'}(\bb T^{a,b} E)$ is trivial. 

Let $T\in\repo AG$ be such that $\te_Q(T)$ belongs to $\langle\ce\,;\,\g T_\ze\rangle_\ot$. If $V\in\repo AG$ is the target of an epimorphism $T\to V$,  exactness of $\te_{Q'}$ produces an epimorphism  
\[\co_{X'}^{\oplus r}\simeq\te_{Q'}(T)\aro\te_{Q'}(V).\] 
Since $\te_{Q'}(V)\ot_A k$ and $\te_{Q'}(V)\ot_A K$ become trivial when pulled back via  $Q'\ot_A k\to X'\ot_Ak$ and $Q'\ot_A K\to X'\ot_AK$, we may apply Lemma \ref{24.05.2018--1}  to  conclude that $\te_{Q'}(V)\ot k$ and $\te_{Q'}(V)\ot K$ are trivial.  Because $X'$ is $H^0$-flat over $A$, Lemma \ref{03.05.2018--1} says  that $\te_{Q'}(V)$ is equally trivial and hence that $\te_Q(V)\in\langle\ce\,;\,\g T_\ze\rangle_\ot$.
By the same argument, now applied to $\check T$, we conclude that $\te_Q(W)$ belongs to $\langle\ce\,;\,\g T_\ze\rangle_\ot$ once $W\to T$ is a special subobject (we employ \cite[Definition 10]{dos_santos09}). 

Now we know that for every  $U\in\repo AG$, there exists a special monomorphism $V\to \bb T^{a,b}E$ and an epimorphism $V\to U$ \cite[Proposition 12]{dos_santos09}. From what was proved above, $\te_Q(U)$ belongs to $\langle\ce\,;\,\g T_\ze\rangle_\ot$. 

To end the proof, let $M\in\rep AG$ be arbitrary. Using \cite[Corollary 2.2, p.41]{serre68}, we find an equivariant presentation \[0\aro U_1\aro U_0\aro M\aro0\]
with $U_0$ and $U_1$ in $\repo AG$.  An application of  Lemma \ref{27.04.2017--3} (and the exactness of the functor $\te_Q$) assures that $\te_Q(M)$ belongs to $\g T_\ze$, and hence to $\langle\ce\,;\,\g T_\ze\rangle_\ot$.  

\bigskip\noindent  (2). We pick $\ze$ and $\ce$ as in item (1) and define  $H={\rm Gal}'(\ce;\g T_X ,x_0)$ so that $\bullet|_{x_0}:\langle\ce\,;\,\g T_X \rangle_\ot\to\rep AH$ is an equivalence of tensor categories. Because of Theorem \ref{25.03.2019--1}, $H$ is a \emph{finite} group scheme over $A$. Using the $A$-point $q_0$ of $Q$ above $x_0$, the functor $(\bullet|_{x_0})\circ\te_Q:\rep AG\to\modules A$ is naturally isomorphic to the forgetful functor and hence 
we derive a morphism of group schemes 
\[
\rho:H\aro G
\]
such that $(\bullet|_{x_0})\circ \te_Q\simeq \rho^\#$.
Now, similarly to \cite{nori} (see \S2 and the argument on p.39), there exists a $H$-torsor $R\to X$ with an $A$-point $r_0$ above $x_0$ such that $\te_{R}\circ(\bullet|_{x_0})\simeq{\rm id}$ and $(\bullet|_{x_0})\circ\te_{R}\simeq{\rm id}$ as tensor functors. (The quasi-coherent  $\co_X$-                                                              algebra of the torsor $R$ corresponds, in the category $\rep A{H}$, to $A[H]_{\rm left}$, see \cite[Definition, p. 32]{nori}.) Now 
\[
\begin{split}\te_Q&\simeq\te_{R}\circ(\bullet|_{x_0})\circ\te_Q\\&\simeq \te_R\circ\rho^\#
\end{split}
\]
which shows, just as in \cite[Propsoition 2.9(c), p. 34]{nori}, that $Q\simeq R\ti^{H}G$. 

To verify the last statement, we note $\rho$ can be decomposed into $H\to H'\stackrel{\si}\to G$, where $\si$ is a closed immersion and $H'$ is finite. This being so, we have $Q\simeq (R\ti^{H}H')\ti^{H'}G$.

\noindent(3). Let $H\in(\bb{FGSch}/A)$ be  finite, $Q\to X$ be an $H$-torsor, $H\to G$ be a closed immersion,  and  $R\ti^HG\simeq Q$ be an isomorphism  as in (2). 
Now, employing the arrow \[R\ti G\aro G,\quad (r,g)\longmapsto g^{-1}\]
we obtain an injection $\mm{Mor}(G,\mathds A^1)\to\mm{Mor}(R\ti G,\mathds A^1)$ and hence an injection 
\[
\left\{\text{$H$-equivariant $G\to\mathds A^1$}\right\}\aro \left\{\text{$H$-equivariant $R\ti G\to\mathds A^1$}\right\}. 
\]
Since the right-hand-side above is simply the ring of     functions of  $R\ti^HG\simeq Q$, the hypothesis then forces $A=A[G]^H$. But $A[G]$ is a finite and locally free $A[G]^H$-module whose rank equals that of $A[H]$ (see III.2.4 of \cite{demazure-gabriel70}). It is then easy to see that the closed immersion $H\to G$ is an isomorphism. 
\end{proof}

We shall now gather some consequences of Theorem \ref{25.06.2018--1} and in doing so connect it to \cite[Chapter II]{nori82} and \cite{antei-emsalem-gasbarri18}. {\it Notations are as in the statement of Theorem \ref{25.06.2018--1}.} 

Let us abbreviate $\Pi=\Pi(X,x_0)$.
As already explained in Section \ref{15.03.2019--2} (see the discussion preceding Corollary  \ref{27.03.2019--1}), there exists a $\Pi$-torsor 
\[
\wt X\aro X
\] 
with an $A$-point $\tilde x_0$ above $x_0$ such that $\te_{\wt X}\circ(\bullet|_{x_0})\simeq{\rm id}$ and $(\bullet|_{x_0})\circ\te_{\wt X}\simeq{\rm id}$ as tensor functors. Recall that for each homomorphism $\rho:\Pi\to G$, the fpqc sheaf of the contracted product  $\wt X\ti^\rho G$  (see \cite[III.4.3.2, p.368]{demazure-gabriel70} or \cite[Part I, 5.14]{jantzen87}) has the following description: it  is the quotient of $\wt X\ti G$ by the right action of $\Pi$ defined, on the level of points, by 
\begin{equation}\label{12.03.2019--2}
(\tilde x,g)\po\ga=(\tilde x\ga,\rho(\ga)^{-1}g).
\end{equation}
Let us write \[\chi_\rho:\wt X\ti G\aro \wt X\ti^\rho G\] for the canonical quotient  morphism. Then,  the arrow
\[
(\mm{pr}_{\wt X}, \chi_\rho ):\wt X\ti G\aro \wt X\ti_X(\wt X\ti^\rho G)
\]
is an isomorphism of $G$-torsors over $\wt X$     (see \cite[III.4.3.1]{demazure-gabriel70} or \cite[Part 1, 5.14(3)]{jantzen87}). 
In addition, if we let $\Pi$ act on (the right of) $\wt X\ti G$ as implied by \eqref{12.03.2019--2} and on $\wt X\ti_X(\wt X\ti^\rho G)$ by the action solely on $\wt X$, then
$(\mm{pr}_{\wt X}, \chi_\rho )$ is $\Pi$-equivariant, as a simple verification shows. 

Now we note that  $A=H^0(\co_{\wt X})$  because $A[\Pi]_{\rm left}$, which corresponds to the quasi-coherent $\co_{X}$-module $\co_{\wt X}$, has only constant invariants and because of \cite[Exercise II.1.11, p.67]{hartshorne77}. Since $G$ is affine, any morphism of schemes $\wt X\to G$ must factor through the structural morphism $\wt X\to\spc A$ \ega{I}{}{2.2.4,p.99}; 
we conclude that any arrow $\be: \wt X\ti G\to\wt X\ti G$
between $G$-torsors   must be of the form $(\tilde x,g)\mapsto(\tilde x,cg)$, where $c\in G(A)$. 
This being so, if $\be$ in addition fixes the $A$-point $(\tilde x_0,e)$, we see that $\be=\mm{id}$. 
This has the following pleasing consequence (implicit in \cite[Proposition 3.11]{nori}): 

\begin{lem}\label{12.03.2019--1}Let $G\in\bb{FGSch}/A$ and let 
 $\rho:\Pi\to G$ and $\si:\Pi\to G$ be arrows of group schemes over $A$. 
Let  $\al:\wt X\ti^\rho G\to\wt X\ti^\si G$ be a morphism of $G$-torsors  sending  $\chi_\rho(\tilde x_0,e)$ to $\chi_\si(\tilde x_0,e)$. Then $\rho=\si$ and $\al=\mm{id}$. 
\end{lem}

\begin{proof}We consider the following commutative diagram 
\[
\xymatrix{
\wt X\ti G\ar[d]^\sim_{(\mm{pr},\chi_\rho)}\ar[rr]^{\wt \al}&&  \wt X\ti G\ar[d]_\sim^{(\mm{pr},\chi_\si)}
\\
\wt X\ti_X(\wt X\ti^\rho G)\ar[rr]_{\mm{id}\ti\al}&& \wt X\ti_X(\wt X\ti^\si G).
}
\]
Since $\al$ takes $\chi_\rho(\tilde x_0,e )$ to $\chi_\si(\tilde x_0,e)$, we conclude that $\wt\al(\tilde x_0,e)=(\tilde x_0,e)$  and hence, by the above discussion, $\wt \al$ is the identity. Because $(\mm{pr},\chi_\rho)$ is $\Pi$-equivariant (for the actions explained above), we conclude that $\wt \al={\rm id}$ is $\Pi$-equivariant, and this is only possible when $\rho=\si$. Since   $\mm{id}\ti\al=\mm{id}$, fpqc descent \ega{IV}{2}{2.7.1,p.29} assures that $\al=\mm{id}$. 

\end{proof}

 Let $G\in(\bb{FGSch}/A)$ be {\it quasi-finite} and consider the category $\bb{Tors}_*(G)$ whose 
\begin{enumerate}\item [\textbf{objects}] are couples $(Q,q_0)$ consisting of a $G$-torsor over $X$ and an $A$-point $q_0$ of $Q$ above $x_0$, and
\item[\textbf{arrows}] are  isomorphisms of $G$-torsors which preserve the $A$-rational point.  
\end{enumerate}

\begin{thm}\label{13.03.2019--2}We maintain the above notations.   
\begin{enumerate}[(1)]
\item[(1)]  Write $|\bb{Tors}_*(G)|$ for the set of isomorphism classes in $\bb{Tors}_*(G)$ and   $\wt X[\rho]$ for the   class of the couple $(\wt X\ti^\rho G,\chi_\rho(\tilde x_0,e))$. Then the map \[
\hh{}{\Pi}{G}\aro  |\bb{Tors}_*(G)|,\qquad\rho\longmapsto\wt X[\rho]
\]
is bijective. 

\item[(1')] The category $\bb{Tors}_*(G)$ is discrete \cite[p.11]{maclane98}. 

\item[(2)] For each $(Q,q_0)$ in $\bb{Tors}_*(G)$, there exists a \emph{unique} generalized morphism of torsors $\chi:\wt X\to Q$ taking $\tilde x_0$ to $q_0$. 
\end{enumerate}

\end{thm}

\begin{proof}(1) We first establish surjectivity: the argument is identical to the one in \cite[Proposition 3.11]{nori}.  
From  Theorem \ref{25.06.2018--1}-(1), the functor $\te_Q$ takes values in $\g T_X^{\rm tan}$.  Using the $A$-point $q_0$ of $Q$ above $x_0$, we see that  $(\bullet|_{x_0})\circ\te_Q:\rep AG\to\modules A$ is naturally isomorphic to the forgetful functor and, since $\bullet|_{x_0}:\g T_X^{\rm tan}\stackrel\sim\to\rep A{\Pi}$, 
we derive a morphism of group schemes 
\[
\rho:\Pi \aro G
\]
such that $(\bullet|_{x_0})\circ \te_Q\simeq \rho^\#$.
Now 
\[
\begin{split}\te_Q&\simeq\te_{\wt X}\circ(\bullet|_{x_0})\circ\te_Q\\&\simeq \te_{\wt X}\circ\rho^\#
\end{split}
\]
which shows, just as in \cite[Propsoition 2.9(c), p. 34]{nori}, that $Q\simeq {\wt X}\ti^{\rho}G$. 
Injectivity is a direct consequence of Lemma \ref{12.03.2019--1}. 

 (1') This is a direct consequence of (1) and Lemma \ref{12.03.2019--1}.

 (2) This is standard, but we run the argument for the convenience of the reader. 
Let $\tau:\wt X\to Q$ and $\tau':\wt X\to Q$ be morphisms as in the statement covering morphisms $\rho:\Pi\to G$ and $\rho':\Pi\to G$ respectively. We therefore  deduce arrows   $\ov\tau:\wt X\ti^\rho G\to  Q$ and $\ov\tau':\wt X\ti^{\rho'}G\to Q$ in $\bb{Tors}_*(G)$ such that 
\[
\wt X\arou{(\mm{id},e)} \wt X\ti G\arou{\chi_\rho}\wt X\ti^\rho G\arou{\ov\tau} Q
\] 
and 
\[
\wt X\arou{(\mm{id},e)} \wt X\ti G\arou{\chi_{\rho'}}\wt X\ti^{\rho'} G\arou{\ov\tau'} Q
\]
are respectively $\tau$ and $\tau'$. 
 This produces an arrow $\al:\wt X\ti^\rho G\to \wt X\ti^{\rho'}G$ in $\bb{Tors}_*(G)$. According to Lemma \ref{12.03.2019--1}, $\rho=\rho'$ and $\al=\mm{id}$ so that $\tau=\tau'$.
\end{proof}

This allows us to compare $\Pi$ to the fundamental group scheme introduced in \cite{antei-emsalem-gasbarri18}. 
Recall that these authors show the existence of a pro-quasi-finite (see Definition \ref{25.02.2019--1})
flat group scheme $\Pi^\star$, a $\Pi^\star$-torsor $X^\star\to X$ and an $A$-point $x^\star_0$ above $x_0$ enjoying the following universal property. If $G$ is flat and quasi-finite, $Q\to X$ is a $G$-torsor with a point $q_0$ above $x_0$, then there exists a {\it unique} generalized morphism of torsors $X^\star\to Q$ taking $x_0^\star$ to $q_0$. (See the paragraph after the proof of Theorem 5.2 in \cite{antei-emsalem-gasbarri18}.) From Theorem \ref{13.03.2019--2} we have:

\begin{cor}\label{13.03.2019--1}Under the assumptions of Theorem \ref{25.06.2018--1}, there exists a unique generalized isomorphism of torsors $\wt X\to X^\star$ taking $\tilde x_0$ to $x_0^\star$. In particular, $\Pi\simeq \Pi^\star$.  \qed
\end{cor}

In particular, Theorem \ref{25.03.2019--1} says that $\Pi^\star$ is in fact pro-finite. 


\end{document}